\documentclass[10pt, a4paper]{article}

\usepackage{fancyhdr}
\fancyheadoffset{ 0.5 cm}
\fancyfootoffset{ 0.5 cm}

\usepackage{titlesec}
\titlelabel{\thetitle.\quad}

\usepackage{times}
\usepackage{geometry}
\usepackage{graphicx}
\usepackage{amssymb, amscd}
\usepackage{amsmath}
\makeatletter
\@addtoreset{equation}{section}
\makeatother
\numberwithin{equation}{subsection}
\usepackage{amsthm} 
\usepackage{mathrsfs}
\usepackage{epstopdf}
\usepackage{enumerate}
\usepackage{url}
\usepackage{graphicx}
\usepackage{mathtools}
\usepackage{tikz}							
\usepackage{setspace}
\usepackage{amsfonts, amsmath, wasysym}
\usepackage{stackrel}

\usepackage{hyperref}
\usepackage{cite}

\usepackage[final]{optional}

\usepackage{color}

\usepackage{xcolor}
\newtagform{red}{\color{red}(}{)}

\usepackage{mathtools}

\usepackage{amsmath}
\makeatletter
\@addtoreset{equation}{section}
\makeatother
\numberwithin{equation}{section}

\theoremstyle{plain}
\newtheorem{theorem}{Theorem}[section]
\newtheorem{lemma}[theorem]{Lemma}

\theoremstyle{definition}

\theoremstyle{remark}

\newcommand{\cO}{\ensuremath{{\cal O}}}

\newcommand{\m}{\ensuremath{{\cal M}}}
\newcommand{\n}{\ensuremath{{\cal N}}}

\newcommand{\cd}{\ensuremath{{\cal D}}}

\newcommand{\ti}{\tilde}

\newcommand{\al}{\alpha}
\newcommand{\be}{\beta}

\newcommand{\de}{\delta}

\newcommand{\Om}{\Omega}

\newcommand{\Si}{\Sigma}
\renewcommand{\th}{\theta}

\newcommand{\vph}{\varphi}
\newcommand{\ep}{\varepsilon}

\newcommand{\N}{\ensuremath{{\mathbb N}}}
\newcommand{\B}{\ensuremath{{\mathbb B}}}

\newcommand{\downto}{\downarrow}

\newcommand{\grad}{\nabla}

\DeclareMathOperator{\VolBB}{Vol\mathbb{B}}

\DeclareMathOperator{\inj}{inj}

\def\blbox{\quad \vrule height7.5pt width4.17pt depth0pt}

\newcommand{\beq}{\begin{equation}}
\newcommand{\eeq}{\end{equation}}
\newcommand{\beqa}{\begin{equation}\begin{aligned}}
\newcommand{\eeqa}{\end{aligned}\end{equation}}
\newcommand{\brmk}{\begin{rmk}}
\newcommand{\ermk}{\end{rmk}}
\newcommand{\partref}[1]{\hbox{(\csname @roman\endcsname{\ref{#1}})}}
\newcommand{\half}{\frac{1}{2}}

\newcommand{\cmt}[1]{\opt{draft}{\textcolor[rgb]{0.5,0,0}{
$\LHD$ #1 $\RHD$\marginpar{\blbox}}}}

\newcommand{\bcmt}[1]{\opt{draft}{\textcolor[rgb]{0,0,0.9}{
$\LHD$ #1 $\RHD$\marginpar{\blbox}}}}

\newcommand{\Rm}{{\mathrm{Rm}}}
\newcommand{\Ric}{{\mathrm{Ric}}}

\usepackage{soul}

\newcommand{\twopartcond}[4]
{ \setstretch{1.5}
	\left\{
		\begin{array}{ll}
			#1 & \mbox{on } #2 \\ 
			#3 & \mbox{on } #4
		\end{array}
	\right.
}

\newenvironment{claim}[1]{\par\noindent\underline{Claim:}\space#1}{}
\newenvironment{claimproof}[1]{\par\noindent\underline{Proof:}\space#1}{\leavevmode\unskip\penalty9999 \hbox{}\nobreak\hfill\quad\hbox{$\dagger \dagger$}}
	
\title{Global Regularity of Three-Dimensional 
Ricci Limit Spaces}
\author{Andrew D. McLeod and Peter M. Topping}
\date{ \today}

\begin{document}

\usetagform{red}

\maketitle

\begin{abstract}
In their recent work \cite{Topping2}, Miles Simon and the second author established a local bi-H\"{o}lder correspondence between weakly noncollapsed Ricci limit spaces in three dimensions and smooth manifolds. In particular, any open ball 
of finite radius in such a limit space must be 
bi-H\"{o}lder homeomorphic to some open subset of a complete smooth Riemannian three-manifold.
In this work we build on the technology from 
\cite{Topping1, Topping2} to improve this local correspondence to a global-local correspondence. That is, we construct a smooth three-manifold $ M ,$ and prove that the entire (weakly) noncollapsed three-dimensional Ricci limit space is homeomorphic to $M$ via a globally-defined homeomorphism that is bi-H\"{o}lder once restricted to any compact subset. Here the bi-H\"{o}lder regularity is with respect to the distance $d_g$ on $M,$ where $g$ is any smooth complete metric on $M .$  

A key step in our proof is the construction of local \emph{pyramid} Ricci flows, existing on uniform regions of spacetime, that are inspired by Hochard's \emph{partial} Ricci flows
\cite{Hochard}. Suppose $\left( M , g_0 , x_0 \right)$ is a complete smooth pointed Riemannian three-manifold that is (weakly) noncollapsed and satisfies a lower Ricci bound. Then, given any $k \in \N,$ we construct a smooth Ricci flow  $g(t)$ living on a subset of spacetime
that contains, for each  $j \in \left\{1 , \ldots , k \right\}$,
a cylinder $\B_{g_0} \left( x_0 , j \right)\times [0,T_j]$, where $T_j$ is dependent only on the Ricci lower bound, the (weakly) noncollapsed volume lower bound and the radius $j$
(in particular independent of $k$)
and with the property that $g(0)=g_0$ throughout $\B_{g_0} \left( x_0 , k \right)$.
\end{abstract}

{\small \tableofcontents}

\section{Introduction}

Given a sequence of $n$-dimensional complete, smooth, pointed Riemannian manifolds $\left( \m_i , g_i , x_i \right),$ for which $\Ric_{g_i} \geq - \al_0$  for some given $\al_0,$ Gromov's compactness theorem, see \cite{Cheeger} for instance, tells us that, after passing to a subsequence, there exists a locally compact complete pointed metric space $\left(X , d_X , x_0 \right)$ for which $ \left( \m_i , d_{g_i} , x_i \right) \rightarrow ( X , d_X , x_0 )$ in the pointed Gromov-Hausdorff sense; the relevant definition of pointed Gromov-Hausdorff convergence may be found in, for example, either \cite{burago} or \cite{Cheeger}. It is a natural question to ask about the regularity of the limit space $( X , d_X )$,
continuing a long tradition of such results that originates with the study 
of limit spaces of manifolds with uniform lower sectional curvature bounds (see e.g. 
\cite{burago} as a starting point).
In this paper we consider the weakly noncollapsed setting, that is with the added assumption that 
$\VolBB_{g_i} (x_i , 1) \geq v_0 > 0.$ 
We refer to this setting as \emph{weakly} noncollapsed since we only require a \emph{single} unit ball $\B_{g_i} (x_i , 1)$ to have a specified uniform lower volume bound as opposed to the stronger globally noncollapsed condition in which we require \emph{all} balls $\B_{g_i} (x , 1)$ to have the uniform lower volume bound.
This stronger globally noncollapsed hypothesis can be handled using Ricci flow techniques that are far simpler than those required in this paper.


Pioneering regularity results were obtained for the limit spaces $\left( X , d_X , x_0 \right)$ of sequences of $n$-dimensional manifolds with uniform lower Ricci bounds
by Cheeger-Colding, see \cite{Cheeger}, as we now describe. 
In the weakly noncollapsed setting the \emph{`regular set'} $\mathcal{R}$ of $X$ is the set of points in $X$ at which all tangent cones are $n$-dimensional Euclidean space; see \cite{Cheeger}. 
Cheeger-Colding \cite{Cheeger-Colding} proved that while the Hausdorff dimension of $X$ is $n$, the singular set $\mathcal{S} := X \setminus \mathcal{R}$ has Hausdorff dimension no larger than $n-2$, and 
the regular set is contained within an open set that is locally bi-H\"{o}lder homeomorphic to a smooth manifold.

Recently, Miles Simon and the second author obtained improved regularity in dimension three; in \cite{Topping2} it is proved that weakly noncollapsed Ricci limit spaces in dimension three
are topological manifolds throughout the entire limit space, irrespective of singularities. In fact, given any point $x \in X,$ including any singular point, there is a neighbourhood of $x$
that is \emph{bi-H\"{o}lder} homeomorphic to a ball in $\mathbb{R}^3.$ 
Moreover, the theory in that paper establishes that for any 
$r>0$, the ball $\B_{d_X}(x_0,r)$ is bi-H\"older homeomorphic to an open subset in a complete smooth Riemannian manifold.
See Theorem 1.4 and Corollary 1.5 in \cite{Topping2} for full details. 
In this paper we use all the technology from \cite{Topping2} and key results and ideas from \cite{Topping1, Hochard} in order to prove directly the following result.

\begin{theorem}[{\bf \em Ricci limit spaces are globally smooth manifolds}]
\label{initial_main_thm} 
Suppose that $\left( \m^3_i , g_i,x_i \right)$ is a sequence of complete, smooth, pointed Riemannian three-manifolds such that for some $\al_0 >0$ and $v_0 > 0$, and for all $i \in \N,$
we have $\Ric_{g_i} \geq - \alpha_0$ throughout $\m_i,$ and $\VolBB_{g_i}  (x_i , 1)  \geq v_0 > 0.$

Then there exist a smooth manifold $M$, a point $x_0\in M$,
and a complete distance metric $d: M \times M \to [0,\infty)$ generating the same topology as $M$ such that after passing to a subsequence in $i$ we have 
$$\left( \m^3_i , d_{g_i} ,x_i \right)\to\left( M , d , x_0 \right),$$
in the pointed Gromov-Hausdorff sense, and if $g$ is any smooth complete Riemannian metric on $M$ then the identity map 
$(M,d)\to (M,d_g)$ is locally bi-H\"older.
\end{theorem}

\vskip 4pt
\noindent
A key part of \cite{Topping2} is the use of Ricci flow to \emph{`mollify'} the Riemannian manifolds $(\mathcal{M}_i,g_i)$ in the spirit of early work of Simon e.g. \cite{SimonCAG,MilesCrelle3D}.
However, it is not expected that there exists any traditional smooth Ricci flow that starts from a general limit space $(X,d_X)$,
or even from a general \emph{smooth} three-manifold with Ricci curvature bounded below \cite{ICM2014}, 
so in \cite{Topping2} a notion of local Ricci flow is used, which automatically generates not just a Ricci flow, but also the underlying smooth manifold for the flow, see e.g. Theorem 1.1 in \cite{Topping2}. This Ricci flow is posed within a class of flows with good estimates, and it is not reasonable to ask for uniqueness of solutions. A consequence of this is that if one takes a second local Ricci flow on a larger local region of the limit space, then restricts to the original local region, there is no guarantee that the natural identification of the two resulting smooth underlying manifolds will be smooth.
Consequently, Theorem \ref{initial_main_thm} of this paper does not immediately follow.

These considerations encourage us to look again at the idea of trying to imagine a Ricci flow starting from the entire Ricci limit space
$(X,d_X)$. We have already pointed out that this should be impossible in the traditional manner, but it is instructive to imagine why we cannot construct such a Ricci flow as a limit of local Ricci flows that exist on larger and larger balls $\B_{g_i}  (x_i , i)$. The problem is that the degree of noncollapsing of such balls typically degenerates as $i\to\infty$, and therefore the existence time of the corresponding local Ricci flows must necessarily degenerate to zero.

The solution to these problems, refining an approach of Hochard \cite{Hochard}, is to consider Ricci flows that live on a subset of spacetime that is not simply a parabolic cylinder $\m\times [0,T]$.
Given a smooth, complete Riemannian three-manifold $\left( \m , g_0 , x_0 \right)$ satisfying the above Ricci lower bound and weakly noncollapsed condition, then for any $k \in \N,$ we prove the existence of a smooth Ricci flow $g_k(t)$ that is defined on a subset of spacetime that contains, for each $m \in \left\{ 1 , \ldots , k \right\},$ the cylinder $\B_{g_0} (x_0 , m)\times \left[ 0 , T_m \right],$ where crucially $T_m > 0$ depends only on $\alpha_0,$ $v_0$ and $m$, and in particular \emph{not} on $k$. Further, the flow enjoys local curvature bounds
on the set $\B_{g_0} (x_0 , m) \times \left( 0 , T_m \right]$, which again depend only on $\alpha_0,$ $v_0$ and $m$.   

\begin{theorem}[\bf{\emph{Pyramid Ricci flow construction}}]
\label{Ricci Flow} 
Suppose that $\left( M^3 , g_0 \right)$ is a complete smooth Riemannian three-manifold and $x_0 \in M.$ 
Assume that for given $\al_0 , v_0 > 0$ we have both $\Ric_{g_0} \geq - \alpha_0$ throughout $M,$ and $\VolBB_{g_0}  ( x_0 , 1) \geq v_0 > 0.$ 

Then there exist increasing sequences $C_k \geq 1$ and  $\al_k>0$, and a decreasing sequence $T_k> 0$, all defined for $k \in \N$, and  depending only on $\al_0$ and $v_0,$ such that the following is true. 


For any $ k \in \N$ there exists a smooth Ricci flow solution $g_k(t),$ defined on a subset $\cd_k$ of spacetime given by 
$$\cd_k := \bigcup_{m=1}^{k} \mathbb{B}_{g_0} (x_0 , m) \times \left[ 0 , T_m \right],$$
with $g_k(0) = g_0$ on $\mathbb{B}_{g_0} (x_0 , k)$, and satisfying,
for each $m\in \{1,\ldots, k\}$,
\beq\label{conc_1} \twopartcond
		{\Ric_{g_k(t)} \geq -\al_m} { \B_{g_0} ( x_0 , m ) \times \left[0,T_m\right]}
		{\left| \Rm \right|_{g_k(t)} \leq \frac{C_m}{t}} {\B_{g_0} ( x_0 , m ) \times \left(0,T_m\right].}
\eeq
\end{theorem}


\vskip 4pt
\noindent
The domain of definition $\cd_k$ of the Ricci flow $g_k(t)$ has a pyramid structure, as illustrated in the following figure, 
and throughout this work we shall term such Ricci flows as \emph{`Pyramid Ricci flows.'}.

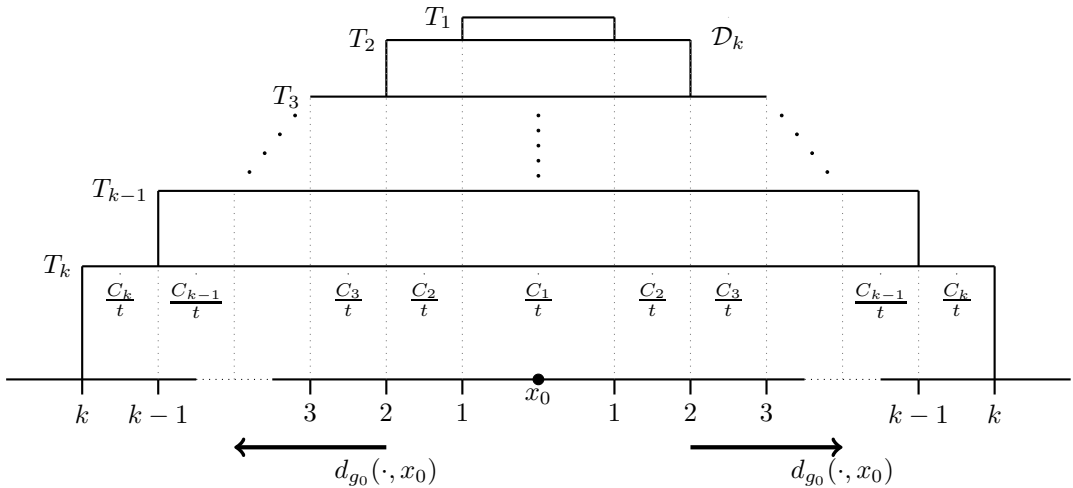
\begin{center}
\begin{tikzpicture}
	\fill[white!25] (0,0)rectangle(1,1.5);  
	\fill[white!25] (11,0)rectangle(12,1.5);  
	\fill[white!25] (1,0)rectangle(2,2.5);  
	\fill[white!25] (10,0)rectangle(11,2.5);
	\fill[white!25] (8,0)rectangle(9,3.75);
	\fill[white!25] (3,0)rectangle(4,3.75);
	\fill[white!25] (4,0)rectangle(5,4.5);
	\fill[white!25] (7,0)rectangle(8,4.5);
	\fill[white!25] (5,0)rectangle(7,4.8);
	\draw[black,thick] (-1,0) -- (1.5,0);
	\draw[black,dotted] (1.5,0) -- (2.5,0);
	\draw[black,thick] (2.5,0) -- (9.5,0);
	\draw[black,thick] (10.5,0) -- (13,0);
	\draw[black,dotted] ( 9.5 , 0 ) -- (10.5,0);
	\filldraw[black] (6,0) circle (2pt) node[anchor=north] {$x_0$};
	\draw[black,thick] (0,1.5)node[anchor=east] {$T_k$} -- (12,1.5);
	\draw[black,thick] (1,2.5)node[anchor=east] {$T_{k-1}$} -- (11, 2.5);
	\draw[black,thick] (3,3.75)node[anchor=east] {$T_3$} -- (9 , 3.75);
	\draw[black,thick] (4,4.5)node[anchor=east] {$T_2$} -- (8,4.5);
	\draw[black,thick] (5,4.8)node[anchor=east] {$T_1$} -- (7,4.8);
	\filldraw[black] (6, 2.7) circle (0.5pt);
	\filldraw[black] (6, 2.9) circle (0.5pt);
	\filldraw[black] (6, 3.1) circle (0.5pt);
	\filldraw[black] (6,3.3) circle (0.5pt);
	\filldraw[black] (6,3.5) circle (0.5pt);
	\filldraw[black] (2.2,2.75) circle (0.5pt);
	\filldraw[black] (2.4,3) circle (0.5pt);
	\filldraw[black] (2.6,3.25) circle (0.5pt);
	\filldraw[black] (2.8, 3.5) circle (0.5pt);
	\filldraw[black] (9.8,2.75) circle (0.5pt);
	\filldraw[black] (9.6,3) circle (0.5pt);
	\filldraw[black] (9.4,3.25) circle (0.5pt);
	\filldraw[black] (9.2, 3.5) circle (0.5pt);
	\draw[black,thick] (0,0) -- (0,1.5);
	\draw[black,thick] (12,0) -- (12,1.5);
	\draw[black,thick] (1,1.5) -- (1, 2.5);
	\draw[black,thick] (11,1.5) -- (11,2.5);
	\draw[black,thick] (4,3.75) -- (4,4.5);
	\draw[black,thick] (8,3.75) -- (8,4.5);
	\draw[black,thick] (5,4.5) -- (5,4.8);
	\draw[black,thick] (7,4.5) -- (7,4.8);
	\draw[gray,dotted] (5,0) -- (5,4.8);
	\draw[gray,dotted] (7,0) -- (7,4.8);
	\draw[gray,dotted] (4,0) -- (4,4.5);
	\draw[gray,dotted] (8,0) -- (8,4.5);
	\draw[gray,dotted] (3,0) -- (3,3.75);
	\draw[gray,dotted] (9,0) -- (9,3.75);
	\draw[gray,dotted] (2,0) -- (2,2.5);
	\draw[gray,dotted] (10,0) -- (10,2.5);
	\draw[gray,dotted] (1,0) -- (1,1.5);
	\draw[gray,dotted] (11,0) -- (11,1.5);
	\draw[ultra thick, ->] (8,-.9) -- (10,-.9) node[anchor=north]{$d_{g_0} ( \cdot , x_0)$};
	\draw[ultra thick, <-] (2,-.9) -- (4,-.9) node[anchor=north]{$d_{g_0} ( \cdot , x_0)$};
	\draw[black,thick] (7,0) -- (7,-.2) node[anchor=north]{$1$};
	\draw[black,thick] (8,0) -- (8,-.2) node[anchor=north]{$2$};
	\draw[black,thick] (11,0) -- (11,-.2) node[anchor=north]{$k-1$};
	\draw[black,thick] (12,0) -- (12,-.2) node[anchor=north]{$k$};
	\draw[black,thick] (3,0) -- (3,-.2) node[anchor=north]{$3$};
	\draw[black,thick] (9,0) -- (9,-.2) node[anchor=north]{$3$};
	\draw[black,thick] (5,0) -- (5,-.2) node[anchor=north]{$1$};
	\draw[black,thick] (4,0) -- (4,-.2) node[anchor=north]{$2$};
	\draw[black,thick] (1,0) -- (1,-.2) node[anchor=north]{$k-1$};
	\draw[black,thick] (0,0) -- (0,-.2) node[anchor=north]{$k$};
	\coordinate (G) at (8.5,4.8);
	\filldraw[black] (G) circle (0.005pt)node[anchor=north]{$\mathcal{D}_k$};
	\filldraw[black] (6,1.4) circle (.05pt)node[anchor=north]{$\frac{C_1}{t}$};
	\filldraw[black] (7.5,1.4) circle (.05pt)node[anchor=north]{$\frac{C_2}{t}$};
	\filldraw[black] (4.5,1.4) circle (.05pt)node[anchor=north]{$\frac{C_2}{t}$};
	\filldraw[black] (3.5,1.4) circle (.05pt)node[anchor=north]{$\frac{C_3}{t}$};
	\filldraw[black] (8.5,1.4) circle (.05pt)node[anchor=north]{$\frac{C_3}{t}$};
	\filldraw[black] (1.5,1.4) circle (.05pt)node[anchor=north]{$\frac{C_{k-1}}{t}$};
	\filldraw[black] (10.5,1.4) circle (.05pt)node[anchor=north]{$\frac{C_{k-1}}{t}$};
	\filldraw[black] (0.5,1.4) circle (.05pt)node[anchor=north]{$\frac{C_k}{t}$};
	\filldraw[black] (11.5,1.4) circle (.05pt)node[anchor=north]{$\frac{C_k}{t}$};
\end{tikzpicture}
\end{center}
As the distance from the central point $x_0$ increases, not only does the existence time of the flow decrease, but the $C/t$ curvature decay estimate worsens. This is in contrast to the \emph{partial} Ricci flow construction of Hochard, and is essential to obtain the uniform estimates on the domain of existence. Another distinction to partial Ricci flows is that by virtue of the theory of Miles Simon and the second author in \cite{Topping1, Topping2}, in particular the so-called Double Bootstrap lemma, our flows have lower Ricci bounds that do not degenerate as $t\downto 0$.
These uniform lower Ricci bounds will be crucial for obtaining our bi-H\"{o}lder estimates in Theorem \ref{initial_main_thm}, and to make the application to Ricci limit spaces,
thanks to the bi-H\"older regularity from 
Lemma 3.1 in \cite{Topping2} (see Lemma \ref{bi-holder distance estimates} of this paper).

Although we do not need it to prove Theorem \ref{initial_main_thm}, we record now that
our pyramid Ricci flows constructed in Theorem \ref{Ricci Flow} allow us to prove 
the following hybrid of the local and global existence results from \cite{Topping2}.


\begin{theorem}[\bf{\emph{Global-Local Ricci flows}}]
\label{mollification}
Suppose that $\left( M , g_0 , x_0 \right)$ is a complete, smooth, pointed, Riemannian three-manifold and, for given $\al_0 , v_0 > 0,$ we have both $\Ric_{g_0} \geq - \al_0$ throughout $M,$ and $\VolBB_{g_0} ( x_0 , 1 ) \geq v_0 > 0.$
Then there exist increasing sequences $C_j \geq 1$ and $\al_j > 0$ and a decreasing sequence $T_j > 0,$ all defined for $j \in\N$, and depending only on $\al_0$ and $v_0$, for which the following is true.

There exists a smooth Ricci flow $g(t),$ defined on a subset of spacetime that contains, for each $j \in \N,$ the cylinder $\B_{g_0} ( x_0 , j) \times \left[0, T_j \right],$ 
satisfying that $g(0)=g_0$ throughout $M$, and further that, again for each $j \in \N,$
\beq\label{moll concs} \twopartcond
		{\Ric_{g(t)} \geq -\al_j} {\B_{g_0} ( x_0 , j) \times \left[ 0 , T_j \right]}
		{ \left| \Rm \right|_{g(t)} \leq \frac{C_j}{t}} {\B_{g_0} ( x_0 , j ) \times \left( 0 , T_j \right].}
\eeq
\end{theorem}
\vskip 4pt
\noindent
To reiterate, in this result we only assume weak noncollapsing, and thus we must 
not expect global existence for positive times.

Analogously to Theorem 1.8 from \cite{Topping2}, we can obtain this sort of global-local existence starting also from a \emph{weakly} noncollapsed Ricci limit space, and in doing so we establish most of 
Theorem \ref{initial_main_thm}.


\begin{theorem}\label{rough data} {\bf{(\emph{Ricci flow from a weakly noncollapsed 3D Ricci limit space})}}
Suppose that $\left( \m^3_i , g_i ,x_i \right)$ is a sequence of complete, smooth, pointed Riemannian three-manifolds such that 
for given $\al_0 , v_0 > 0$ we have $\Ric_{g_i} \geq - \alpha_0$ throughout $\m_i,$ and $\VolBB_{g_i}  (x_i , 1)  \geq v_0 > 0$, for each $i \in \N$.

Then there exist increasing sequences $C_k \geq 1$ and $\al_k > 0$ and a decreasing sequence $T_k > 0,$ all defined for $k \in\N$, and depending only on $\al_0$ and $v_0,$ for which the following holds.

There exist a smooth three-manifold $M,$ a point $x_0 \in M,$ 
a complete distance metric $d : M \times M \rightarrow [0,\infty)$ generating the same topology as we already have on $M,$ 
and a smooth Ricci flow $g(t)$ 
defined on a subset of spacetime $M \times (0,\infty)$ that contains $\B_{d} ( x_0 , k) \times \left(0, T_k \right]$ for each $k \in \N,$ 
with $d_{g(t)} \rightarrow d$ locally uniformly on $M$ as $t \downarrow 0,$ 
and after passing to a subsequence in $i$ we have that $\left( \m_i , d_{g_i} , x_i \right)$ converges in the pointed Gromov-Hausdorff sense to $\left( M , d , x_0 \right).$ 
Moreover, for any $k \in \N,$ 
\begin{equation}\label{r_d_concs} \twopartcond 
		{\Ric_{g(t)} \geq -\al_k} {\B_{d} ( x_0 , k ) \times \left(0,T_{k}\right]}
		{\left| \Rm \right|_{g(t)} \leq \frac{C_{k}}{t}} {\B_{d} ( x_0 , k ) \times \left(0,T_{k}\right].}
\end{equation}
\end{theorem}


\vskip 4pt
\noindent
This theorem will be a special case of the more elaborate Theorem \ref{overarching_thm} that will explicitly arrive at $g(t)$ as a limit of pyramid Ricci flows via pull-back by diffeomorphisms 
generated by a local form of Hamilton-Cheeger-Gromov compactness that we give in Lemma 
\ref{loc_compactness_flows}.
A further special case of Theorem \ref{overarching_thm}
will be Theorem \ref{initial_main_thm}, and the following stronger assertion.



\begin{theorem}[\bf \em Regular GH approximations]
\label{extended_initial_main_thm}
In the setting of Theorem \ref{initial_main_thm}, we may assume the following additional conclusions: 

There exists a sequence of smooth maps $\vph_i:\B_d(x_0,i)\to \m_i$, diffeomorphic onto their images, and mapping $x_0$ to $x_i$ such that
for any $R>0$ we have
$d_{g_i}(\vph_i(x),\vph_i(y))\to d(x,y)$ uniformly for 
$x,y\in \B_d(x_0,R)$ as $i\to\infty$.

Moreover, for sufficiently large $i$, $\vph_i|_{\B_d(x_0,R)}$ is
bi-H\"older with H\"older exponent depending only on $\al_0$, $v_0$ and $R$.

Finally, for any $r\in (0,R)$, and for sufficiently large $i$,
$\vph_i|_{\B_d(x_0,R)}$ maps onto $\B_{g_i}(x_i,r)$.
\end{theorem}

\noindent
Thus, not only do we have the pointed Gromov-Hausdorff convergence of Theorem \ref{initial_main_thm}, we can also find Gromov-Hausdorff approximations that are smooth and bi-H\"older (neglecting a thin boundary layer) cf. Theorem 1.4 from \cite{Topping2}.

In this paper we utilise numerous results, and mild variants of results, from \cite{Topping1} and \cite{Topping2}. For convenience we collect all such material in Appendix \ref{appA}. 
There are also several substantial deviations from existing theory. 
The main novelty is the new \emph{pyramid extension lemma} \ref{PEL}. 
This result asserts that it is not just possible to construct a local Ricci flow with good estimates, but that we can additionally assume that this local flow extends a given Ricci flow defined for a shorter time on a larger domain. 
The estimates, and their constants, are handled with sufficient care that the pyramid extension lemma can be iterated, 
in Section \ref{proof of ricci flow}, to 
construct the pyramid Ricci flows of Theorem \ref{Ricci Flow}.
Another notable difference between our work and existing theory arises in the Ricci flow compactness  of Section \ref{step}. 
For compactness of pyramid flows we must appeal to compactness of the flows not at one time slice, as in the traditional theory, but at countably many time slices. 
The resulting Theorem \ref{overarching_thm} in turn establishes Theorems \ref{initial_main_thm}, \ref{extended_initial_main_thm} and \ref{rough data}.

\vskip10pt

\noindent
\emph{Acknowledgements:} 
This work was supported by EPSRC grant number EP/K00865X/1 and an EPSRC doctoral fellowship EP/M506679/1. The second author would like to thank Miles Simon for many discussions on this topic.

\section{The Pyramid Extension Lemma}\label{constants}

The following result interpolates between the local existence theorem (Theorem 1.6) and the extension lemma (Lemma 4.4) of Simon-Topping \cite{Topping2}, and is the major ingredient in constructing pyramid Ricci flows.

\begin{lemma}[\bf \em Pyramid Extension Lemma]
\label{PEL}
Suppose $(M,g_0,x_0)$ is a pointed complete Riemannian 3-manifold satisfying $\Ric_{g_0}\geq -\al_0<0
$ throughout $M,$ and $\VolBB_{g_0}(x_0,1)\geq v_0>0$. 
Then there exist  increasing sequences
$C_k\geq 1$ and $\al_k>0,$ 
and a decreasing sequence 
$S_k> 0$, all defined for $k\in\N$ and depending only on $\al_0$ and $v_0$, with the following properties. 

First, for each $k\in \N$ there exists a Ricci flow $g(t)$ on 
$\B_{g_0}(x_0,k)$ 
for $t\in [0,S_k]$ such that $g(0)=g_0$ where defined and so that 
$|\Rm|_{g(t)}\leq C_k/t$ for all $t\in (0,S_k]$ and $\Ric_{g(t)}\geq -\al_k$ for all $t\in [0,S_k]$. 

Moreover, given any Ricci flow $\tilde g(t)$ on 
$\B_{g_0}(x_0,k+1)$ over a time interval $t\in [0,S]$ with $\tilde g(0)=g_0$ where defined, 
and with $|\Rm|_{\tilde g(t)}\leq c_0/t$ 
for some $c_0>0$ and all $t\in (0,S]$, 
there exists $\tilde S_k>0$ depending on $k$, $\al_0,$ $v_0$ and $c_0$ only, 
such that we may choose the Ricci flow $g(t)$ above to agree with the restriction of $\tilde g(t)$ to $\B_{g_0}(x_0,k)$ 
for times $t\in [0,\min\{S,\tilde S_k, S_k\}]$. 
\end{lemma}




\begin{proof}[Proof of Lemma \ref{PEL}]
By making a uniform parabolic rescaling (scaling distances by a factor of $14$), it suffices to prove the lemma under the apparently stronger hypothesis that $\tilde g(t)$ is assumed to be defined not just on $\B_{g_0}(x_0,k+1)$ but on the larger ball $\B_{g_0}(x_0,k+14)$, still satisfying the curvature decay $|\Rm|_{\tilde g(t)}\leq c_0/t$. 

By Bishop-Gromov, for all $k\in\N$, there exists $v_k>0$ depending only on $k$, $\al_0$ and $v_0$ such that 
if $x\in \B_{g_0}(x_0,k+14)$ and $r\in (0,1]$ then $\VolBB_{g_0}(x,r)\geq v_k r^3$.

The first part of the lemma, giving the initial existence statement for $g(t)$, follows immediately by the local existence theorem \ref{local existence} 
for some $C_k\geq 1$, $\al_k>0$ and $S_k>0$ depending only on $\al_0$ and $v_k$, i.e. on 
$\al_0$, $k$ and $v_0$. We will allow ourselves to increase 
$C_k$ and $\al_k$, and decrease $S_k$, in order to establish the remaining claims of the lemma. 

We increase each $C_k$ to be at least as large as the constant $C_0$ retrieved from Lemma \ref{local lemma 1} 
with $v_0$ there equal to $v_k$ here. Note that we are not actually applying Lemma \ref{local lemma 1}, but simply retrieving a constant in preparation for its future application.
By inductively replacing $C_k$ by $\max\{C_k,C_{k-1}\}$ for $k=2,3,\ldots$, we can additionally assume that $C_k$ is an increasing sequence. 
Thus $C_k$ still depends 
only on $k$, $\al_0$ and $v_0$, and in particular, not on $c_0$, and can be fixed for the remainder of the proof.

Suppose now that we would like to extend a Ricci flow $\tilde g(t)$.
Appealing to the double bootstrap lemma \ref{double bootstrap} 
centred at each $x\in \B_{g_0}(x_0,k+12)$, 
there exists $\hat S>0$ depending only on $c_0$ and $\al_0$ so that for all $t\in [0,\min\{S,\hat S\}]$ 
we have $\Ric_{\tilde g(t)}\geq -100\al_0c_0$ throughout 
$\B_{g_0}(x_0,k+12)$. 
(In due course, we will require a lower Ricci bound that does not depend on $c_0$.) 
In addition, after reducing $\hat S>0$, still depending only on
$c_0$ and $\al_0$, the shrinking balls lemma \ref{nested balls} tells us that 
for all $x\in \B_{g_0}(x_0,k+10)$ we have 
$\B_{\tilde g(t)}(x,1)\subset \B_{g_0}(x,2) \subset \B_{g_0}(x_0,k+12)$ 
where the Ricci curvature is controlled, for all 
$t\in [0,\min\{S,\hat S\}]$.

Thus, for $x\in\B_{g_0}(x_0,k+10)$ we can apply Lemma \ref{local lemma 1} to deduce that 
$|\Rm|_{\tilde g(t)}(x)\leq C_k/t$ for all 
$t\in (0,\min\{S,\tilde S_k\}]$, 
for some $\tilde S_k\in (0,\hat S]$ depending only on $v_k$, $\al_0$ and $c_0$, i.e only on $k$, $c_0$, $v_0$ and $\al_0$. 

Now we have a curvature decay estimate that does not depend on $c_0$ (albeit for a time depending on $c_0$) we can return to the double bootstrap lemma \ref{double bootstrap}, 
which then tells us that on the smaller ball 
$\B_{g_0}(x_0,k+8)$ we have $\Ric_{\tilde g(t)}\geq -\al_k$ for 
$t\in [0,\min\{S,\tilde S_k\}]$, 
where $\al_k$ is increased to be at least $100\al_0 C_k$ and will be increased once more below 
(but only ever depending on $k$, $\al_0$ and $v_0$) 
and where we have reduced $\ti S_k>0$ without adding any additional dependencies. 

We can also exploit these new estimates to get better volume bounds via Lemma \ref{Volume 2}. 
We apply that result with $R=k+8$ 
to obtain that for every 
$t\in [0,\min\{S,\tilde S_k\}]$, where we have reduced $\ti S_k>0$ 
again without adding any additional dependencies, 
we have 
$\B_{\ti g(t)}(x_0,k+7)\subset \B_{g_0}(x_0,k+8)$, 
and for every $x\in \B_{\ti g(t)}(x_0,k+6)$, we have 
$\VolBB_{\ti g(t)}(x,1)\geq \ep_k>0$, 
where 
$\ep_k$ depends only on $v_0$, $k$, and $\al_0$.

We need one final reduction of $\ti S_k>0$ in order to ensure appropriate nesting of balls defined at different times. 
By the expanding balls lemma \ref{expanding balls}, 
exploiting our lower Ricci bounds (even the weaker bound suffices here) 
we deduce that
\beq
\label{ball_inclusion_4_5}
\left\{
\begin{aligned}
\B_{g_0}(x_0,k+4)\subset \B_{\ti g(t)}(x_0,k+5)\\
\B_{g_0}(x_0,k+2)\subset \B_{\ti g(t)}(x_0,k+3)
\end{aligned}
\right.
\qquad
\text{ for all }
t\in [0,\min\{S,\tilde S_k\}],
\eeq
with $\ti S_k>0$ reduced appropriately, without additional dependencies.

At this point 
we can fix $\ti S_k$ and try to find our desired extension $g(t)$ of $\ti g(t)$ by considering 
$\ti g(\tau)$ for $\tau:=\min\{S,\ti S_k\}$ and restarting the flow from there. 
We cannot restart the flow using any variant of Shi's existence theorem 
(as was done in the extension lemma from \cite{Topping2}, for example) 
since we would not have appropriate control on the existence time. 
Instead, we appeal to the local existence theorem \ref{local existence}. 
In order to do so, note that $\ti g(\tau)$ satisfies the estimates 
$\Ric_{\tilde g(\tau)}\geq -\al_k$ 
on $\B_{\ti g(\tau)}(x_0,k+7)\subset\subset 
\B_{g_0}(x_0,k+14)$, 
and $\VolBB_{\ti g(\tau)}(x,1)\geq \ep_k>0$ for each 
$x\in \B_{\ti g(\tau)}(x_0,k+6)$. 

The output of the local existence theorem \ref{local existence}, 
applied with 
$M$ there equal to $\B_{g_0}(x_0,k+14)$ here, with
$g_0$ there equal to $\ti g(\tau)$ here, and with $s_0=k+7$, 
is that after reducing the $S_k>0$ that we happened to find at the start of the proof, 
still depending only on $\al_0$, $k$ and $v_0$, there exists a Ricci flow $h(t)$ on $\B_{\ti g(\tau)}(x_0,k+5)$ 
for $t\in [0,S_k]$, with $h(0)=\ti g(\tau)$ where defined, 
and such that 
$\Ric_{h(t)}\geq -\al_k$ (after possibly increasing $\al_k$, still depending only on 
$\al_0$, $k$ and $v_0$)
and $|\Rm|_{h(t)}\leq c_k/t$, where $c_k$ depends only on 
$\al_0$, $k$ and $v_0$.
By the first inclusion of \eqref{ball_inclusion_4_5}, this flow is defined throughout 
$\B_{g_0}(x_0,k+4)$. 

Define a concatenated Ricci flow on 
$\B_{\ti g(\tau)}(x_0,k+5)\supset\B_{g_0}(x_0,k+4)$ 
for $t\in [0,\tau+S_k]$ by
\begin{equation}
\label{def of g} 
g(t) := \left\{
\begin{aligned}
& { \tilde{g}(t) }\qquad & & {0 \leq t \leq \tau } \\
& {h\left(t- \tau \right)}\qquad & & {\tau < t \leq \tau+S_k }.
\end{aligned}
\right.
\end{equation} 
This already satisfies the required lower Ricci bound 
$\Ric_{g(t)}\geq -\al_k$. 

We claim that after possibly reducing $S_k$, without further dependencies, we have that for all $x\in \B_{g_0}(x_0,k+2)$, there holds the inclusion $\B_{g(t)}(x,1)\subset\subset \B_{\ti g(\tau)}(x_0,k+5)$ where the flow is defined, for all $t\in [0,\tau+S_k]$.
But we already arranged that for 
$x\in \B_{g_0}(x_0,k+2)\subset \B_{g_0}(x_0,k+10)$
we have $\B_{\ti g(t)}(x,1)\subset \B_{g_0}(x,2)$, which in turn is compactly contained in
$\B_{g_0}(x_0,k+4)\subset \B_{\ti g(\tau)}(x_0,k+5)$,  so the claim holds up until time $\tau$.

Thus to prove the claim it remains to show that 
for all $x\in \B_{g_0}(x_0,k+2)$, there holds the inclusion 
$\B_{h(t)}(x,1)\subset\subset \B_{h(0)}(x_0,k+5)$ for all $t\in [0,S_k]$,
and by the second inclusion of \eqref{ball_inclusion_4_5}, it suffices to prove this 
for each $x\in \B_{h(0)}(x_0,k+3)$.
But by the shrinking balls lemma \ref{nested balls}, after reducing $S_k$ depending on $c_k$, and thus on $\al_0$, $k$ and $v_0$, we can 
deduce that 
$\B_{h(t)}(x,1)\subset\B_{h(0)}(x,2) \subset\subset\B_{h(0)}(x_0,k+5)$
as required, thus proving the claim.

At this point we truncate the flow $g(t)$ to live only on the time interval $[0,S_k]$ (i.e. we chop off an interval of length $\tau$ from the end, not the beginning). 
The flow now lives on a time interval of length independent of $c_0$ and $S$. 

The main final step is to apply Lemma \ref{local lemma 1} once more, with $M$ there equal to $\B_{\ti g(\tau)}(x_0,k+5)$ here.
Using the claim we just proved, for every $x\in \B_{g_0}(x_0,k+2)$, 
after a possible further reduction of $S_k>0$, and with $C_k$ as fixed earlier, 
the local lemma \ref{local lemma 1} tells us that 
$|\Rm|_{g(t)}(x)\leq C_k/t$ for all $t\in (0,S_k]$.
We finally have a sequence $S_k$ that does what the lemma asks of it, except for being decreasing. The monotonicity of $S_k$ and $\al_k$ can be arranged by
iteratively replacing $S_k$ by $\min\{S_k,S_{k-1}\}$,
and $\al_k$ by $\max\{\al_k,\al_{k-1}\}$, for $k=2,3,\ldots$.

By restricting $g(t)$ to $\B_{g_0} (x_0, k)$ we are done.
\end{proof}

\section{{Pyramid Ricci Flow Construction - Proof of Theorem \ref{Ricci Flow}}}
\label{proof of ricci flow}

\begin{proof}[Proof of Theorem \ref{Ricci Flow}]
%
For our given constants $\al_0$ and $v_0$, we appeal to Lemma \ref{PEL} for increasing sequences $C_k\geq 1$ and $\al_k>0$, 
and a decreasing sequence $S_k> 0$, all defined for $k\in\N$ and depending only on $\al_0$ and $v_0$. 
Moreover, we construct a sequence $\ti S_k$ as follows. For each $k\in \N$, we appeal to Lemma \ref{PEL} with our given constants $\al_0$ and $v_0$ and with $c_0=C_{k+1}$.
The sequences $C_k \geq 1$ and $\al_k >0$ are suitable for the sequences required by the theorem.

An induction argument is required to get the constants $T_k$. We begin by setting $T_1$ to be $S_1$. 
The inductive step is as follows: Suppose we have picked $T_1,\ldots T_{k-1}$ for any integer $k\geq 2$. 
Then we set $T_k$ to be the minimum of $S_k,$ $\tilde S_{k-1}$ and $T_{k-1}$.

Note that when we pick $T_k$, it depends on $S_k$, i.e. on $k$, $\al_0$ and $v_0$, 
and it also depends on $\tilde S_{k-1}$, i.e. additionally on $C_k$, 
but that itself only depends on $k$, $\al_0$ and $v_0$. 

Fix $l \in \N.$ To construct $g_l(t)$, we appeal to Lemma \ref{PEL} $l$ times.

First we use the first part of that lemma with $k=l$. This initial flow lives on $\B_{g_0}(x_0,l)$ for a time $S_l$, and thus certainly for $T_l$.

Since $T_l \leq \ti S_{l-1}$, we can appeal a second time to the lemma, this time with $k=l-1$, 
in order to extend the flow $g_l(t)$ to the longer time interval $[0,T_{l-1}]$, albeit on the smaller ball $\mathbb{B}_{g_0}(x_0,l-1)$. 

We repeat this process inductively for the remaining values of $k$ down until it is finally repeated for $k=1.$ 
The resulting smooth Ricci flow $g_l(t)$ is now defined, for each $m \in \left\{ 1 , \ldots , l \right\},$ 
on $\B_{g_0} \left( x_0 , m \right)$ over the time interval $t \in \left[0, T_m \right],$ 
still satisfying that $g_l(0) = g_0$ where defined. Moreover,
our repeated applications of Lemma \ref{PEL} provide the estimates
\beq
\label{estimates}
\twopartcond
	{\Ric_{g_l(t)} \geq -\al_m } { \B_{g_0} \left( x_0 , m \right) \times \left[0,T_m\right]}
	{|\Rm|_{g_l(t)} \leq \frac{C_m}{t}} { \B_{g_0} \left( x_0 , m \right) \times \left(0,T_m\right]}
\eeq
for each $m \in \left\{1 , \ldots , l \right\}$, which completes the proof.
\end{proof}

\section{Global-Local Mollification}\label{glob moll}


\begin{proof}[Proof of Theorem \ref{mollification}]
For the $\al_0$ and $v_0$ of the theorem, we begin by retrieving sequences 
$C_j,$ $\al_j$ and $T_j$ from Theorem \ref{Ricci Flow}. 
Our first step is to modify them by throwing away the first two terms of each, 
i.e. replacing the three  sequences by $C_{j+2},$ $\al_{j+2}$ and $T_{j+2}$.

With a view to later applying the shrinking balls lemma \ref{nested balls} 
for each $j \in \N$ we reduce $T_j,$ without additional dependencies; with hindsight, it will suffice to ensure that
\beq
	\label{ball_inc_reduce}
		 T_j < \frac{1}{4 \beta^2 C_j} 
\eeq
where $\beta \geq 1$ is the universal constant arising in the shrinking balls lemma \ref{nested balls}.

For each $i \in \N$ let $g_i(t)$ denote the \emph{pyramid Ricci flow} obtained in Theorem  \ref{Ricci Flow} defined on a subset $\cd_i \subset M \times \left[0,\infty\right)$ 
that now contains $\B_{g_0} ( x_0 , l+2) \times \left[0,T_l \right]$ for each $l \in \left\{ 1 , \ldots , i \right\}$,
having deleted the first two terms of the sequences 
$C_j,$ $\al_j$ and $T_j$.
If we fix $j\in\N$, then for $i \geq j$ the estimates of \eqref{conc_1} hold for $g_i(t)$ on the $g_0$ ball of radius $j+2.$
Consider an arbitrary point $y \in \B_{g_0} ( x_0 , j + 1).$ We have the curvature estimate $|\Rm|_{g_i(t)} \leq \frac{C_j}{t}$ throughout $\B_{g_0} ( y , 1) \times (0, T_j].$
The shrinking balls lemma \ref{nested balls} tells us that
$ \B_{g_i (t)} \left( y , \frac{1}{2} \right) \subset \B_{g_0} \left( y , \frac{1}{2} + \beta \sqrt{C_j T_j} \right) \subset \subset \B_{g_0} ( y , 1)$
for any $t \in [0, T_j],$ 
provided $ \frac{1}{2} + \beta \sqrt{C_j T_j} < 1.$
The restriction \eqref{ball_inc_reduce} ensures this is the case, and hence we establish the curvature bound $|\Rm|_{g_i (t)} \leq \frac{C_j}{t}$
throughout $\B_{g_i(t)} \left( y , \frac{1}{2} \right)$ for any $t \in (0, T_j].$
 
These estimates allow us to repeat the argument of Miles Simon and the second author in Theorem 1.7 in \cite{Topping2} 
and deduce that after passing to a subsequence in $i,$ we have smooth convergence $g_i(t)\to g(t)$,
for some smooth Ricci flow $g(t)$ on $\mathbb{B}_{g_0} ( x_0 , j)$, 
defined for $t \in \left[0,T_j\right],$ with $g(0)=g_0$ 
on $\B_{g_0} ( x_0 , j)$, and satisfying the curvature estimates 
\begin{equation}\label{moll-0} \twopartcond
		{\Ric_{g(t)} \geq -\al_j } { \B_{g_0} ( x_0 , j) \times \left[0,T_j\right]}
		{ \left| \Rm \right|_{g(t)} \leq \frac{C_j}{t}} {\B_{g_0} ( x_0 , j ) \times \left( 0, T_j \right].}
\end{equation}
We can now repeat this process for each $j=1,2,\ldots$ and take a diagonal subsequence to obtain a smooth limit Ricci flow $g(t)$ 
on a subset of spacetime that contains $\B_{g_0} ( x_0 , j ) \times \left[0,T_j\right]$ for each $j \in \N,$ with $g(0)=g_0$ throughout $M$, and satisfying \eqref{moll-0} for every $j \in \N.$ 
\end{proof}

%

\section{Pyramid Ricci Flow Compactness Theorem}
\label{step}

The following overarching theorem effectively includes Theorems \ref{initial_main_thm}, 
\ref{rough data} and \ref{extended_initial_main_thm}. 
The Ricci flows $g_k(t)$ arising here are pyramid Ricci flows coming from Theorem \ref{Ricci Flow}.


\begin{theorem}
\label{overarching_thm}
In the setting of Theorem \ref{rough data}, in addition to the conclusions of that theorem, including the existence of $M$, $x_0$, $d$, $g(t)$ and the sequences $C_k$, $\al_k$ and $T_k$, we may assume also that the following holds.

For each $k\in\N$, there exist Ricci flows $g_k(t)$ 
defined on the subset of $\m_k\times [0,\infty)$ defined by
$$\cd_k := \bigcup_{m=1}^{k} \B_{g_k} (x_k , m+2) \times \left[ 0 , T_m \right],$$
with the properties that $g_k(0)=g_k$ on $\B_{g_k} (x_k , k+2)$ and
\beq
\label{flow_conds_calm_k_a} 
\twopartcond
		{\Ric_{g_k(t)} \geq -\al_m} {\B_{g_k} ( x_k , m+2) \times \left[ 0 , T_m \right]}
		{ \left| \Rm \right|_{g_k(t)} \leq \frac{C_m}{t}} {\B_{g_k} ( x_k , m+2 ) \times \left( 0 , T_m \right],}
\eeq
for each $m\in \{1,\ldots,k\}$.

Moreover, for each $m\in \N$, the flows $g_k(t)$
converge 
to $(B_d(x_0,m),g(t))$ for $t\in (0,T_m]$, in the following sense:
There exists a sequence of smooth maps
$f^m_k:B_d(x_0,m)\to \B_{g_k} ( x_k , m+1)\subset\m_k$, mapping $x_0$ to $x_k$, such that for each $\de\in (0,T_m)$ we have
$(f^m_k)^*g_k(t)\to g(t)$ smoothly uniformly on $B_d(x_0,m)\times [\de,T_m]$. 

Moreover, there exists a sequence of smooth maps 
$\vph_k : \B_{d}(x_0,k)\to \B_{g_k}(x_k,k+1)\subset\m_k$, diffeomorphic onto their images, mapping $x_0$ to $x_k$, such that, for any $R>0,$
as $k \rightarrow \infty$ we have the convergence 
$$ d_{g_k} ( \vph_k (x) , \vph_k (y) ) \rightarrow d(x,y)$$
uniformly as $x,y$ vary over $\B_d (x_0 , R)$,
and
for sufficiently large $k$, $\vph_k|_{\B_d(x_0,R)}$ is
bi-H\"older with H\"older exponent depending only on $\al_0$, $v_0$ and $R$.
Moreover,  for any $r\in (0,R)$, and for sufficiently large $k$,
$\vph_k|_{\B_d(x_0,R)}$ maps onto $\B_{g_k}(x_k,r)$.

Finally, if $g$ is any smooth complete Riemannian metric on $M$ then the identity map 
$(M,d)\to (M,d_g)$ is locally bi-H\"older.
\end{theorem}

\vskip 4pt
\noindent
To clarify, by smooth uniform convergence, we mean uniform $C^l$ convergence for arbitrary $l\in\N$.
We remark that the bi-H\"older assertion for the maps $\vph_k$ in this theorem can be taken with respect to the distance metrics $d$ and $d_{g_k}$, although one could replace $g_k$ by any complete smooth metric.

\begin{proof}[Proof of Theorem \ref{overarching_thm}]
For the $\al_0$ and $v_0$ of the theorem (as in Theorem \ref{rough data}) we begin by retrieving sequences 
$C_j,$ $\al_j$ and $T_j$ from Theorem \ref{Ricci Flow}.

Throughout the proof $\eta := \frac{1}{10} $ will be fixed.
With a view to later applying Lemma \ref{bi-holder distance estimates} and both the expanding and shrinking balls lemmas, for each $j \in \N$ we reduce $T_j,$ without additional dependencies, and with hindsight it will suffice to ensure that
%
\beq
\label{dist_est_reduce_a}
\left\{
\begin{aligned}
&	 (i) \quad (4j + 8) ( 1 - e^{-\al_j T_j} ) < 1 - 8 \beta \sqrt{ C_j T_j} 
	 \qquad(\text{in particular }\beta \sqrt{ C_j T_j}<{\textstyle\frac{1}{8}}) 
	 \text{ and}\\
&	(ii) \quad (j+1) ( e^{\al_j T_j} - 1 ) \leq \eta
\end{aligned}
\right.
\eeq
where $\beta \geq 1$ is the universal constant arising in the shrinking balls lemma \ref{nested balls}.
For $j = 2,3,\ldots$, if necessary we inductively replace $T_j$ by $\min \left\{ T_j , T_{j-1} , \frac{1}{j} \right\}$ to ensure the
monotonicity of the sequence $T_j$ remains, and to force
$T_j \downarrow 0$ as $j \to \infty.$

We modify these sequences further by dropping the first two terms, i.e. by replacing each $C_j,$ $\al_j$ and $T_j$ by
$C_{j+2},$ $\al_{j+2}$ and $T_{j+2}$ respectively. 
This does not affect the monotonicity or dependencies. 
We may fix the values $C_j \geq 1$ and $\al_j > 0$ for each $j \in \N$ for the remainder of the proof. 
Before fixing $T_j,$ we (potentially) reduce the value further.

With a view to appealing to Cheeger-Gromov-Hamilton compactness via Lemma \ref{loc_compactness_flows}, 
we reduce $T_j,$ without additional dependencies, so that the conclusions of Lemma \ref{loc_compactness_flows} for hypotheses 
$R = j + 1,$ $\eta = \frac{1}{10} ,$ $n=3,$ $v=v_0,$ $\al = \al_j$ and $c_0 = C_j$ 
are valid for all times $t \in \left(0,T_j\right].$ 
As above, we may assume that $T_j$ remains monotonically decreasing.
After these reductions, we can now fix the value of $T_j$ for each $j \in \N$ for the remainder of the proof. 

For each $k \in \N$ let $g_k(t)$ denote the smooth \emph{pyramid Ricci flow}, defined on the subset $\cd_k \subset \m_k \times \left[0,\infty\right)$ 
obtained via Theorem \ref{Ricci Flow}. 
That is 
\beq
\label{domain_of_def_g_k_a}
	\cd_k = \bigcup_{m=1}^{k} \B_{g_k} ( x_k , m + 2) \times \left[0, T_m \right].
\eeq
(Recall that we have dropped the first two terms of the sequences, so we can work on a radius $m+2$ rather than $m$.)
In particular, we have $g_k(0)=g_k$ where defined and  for each $m \in \left\{ 1 , \ldots , k \right\}$ we have
\beq
\label{compactness_reqs_a}
	\twopartcond
		{\Ric_{g_k(t)} \geq -\al_m } { \B_{g_k} (x_k , m+2) \times \left[0 , T_m \right]}
		{| \Rm |_{g_k(t)} \leq \frac{C_m}{t}} {\B_{g_k} (x_k , m + 2) \times \left(0, T_m\right].}
\eeq

Fix $m \in \N$. For every $k \geq m$ the flow $g_k(t)$ is defined throughout
$\B_{g_k} (x_k , m+2) \times \left[0,T_m\right].$
Combining \eqref{compactness_reqs_a} with $\VolBB_{g_k} (x_k , m+1) \geq v_0 > 0$ allows us to appeal to Lemma \ref{loc_compactness_flows} 
with $R = m + 1,$ $\eta = \frac{1}{10},$ $n=3,$ $v = v_0,$ $\al=\al_m$ and $c_0 = C_m$ 
to deduce that, 
after passing to a subsequence in $k,$ 
we obtain a smooth three-manifold $\n_m,$ a point $x^m_{\infty} \in \n_m$ 
and a smooth Ricci flow $\hat{g}_m(t)$ on $\n_m \times \left( 0 , T_m \right]$ with the following properties. 
First, for any $t \in (0 , T_m]$ we have the inclusion
\beq
	\label{compact_inc}
		\B_{\hat{g}_m (t)} \left( x^m_{\infty} , m +1 - \eta \right) \subset \subset \n_m.
\eeq
Second, we have
\beq
	\label{size_of_M_m}
		\B_{\hat{g}_m(t)} \left( x^m_{\infty} , m+1 - 2\eta \right) \subset M_m,
\eeq
for all $t \in \left(0, T_m \right]$,
where $M_m$ is the connected component of the interior of 
\beq
	\label{M_m_compactly_contained_a}
		\bigcap_{s \in \left(0 , T_m \right]} \B_{\hat{g}_m(s)} \left( x^m_{\infty} , m + 1 - \eta \right) \subset \n_m
\eeq
that contains $x^m_\infty$.
Combining \eqref{compact_inc} and \eqref{M_m_compactly_contained_a} allows us to conclude that 
\beq
	\label{M_m_compactly_contained}
		M_m \subset \subset \n_m.
\eeq
Moreover, Lemma \ref{loc_compactness_flows} gives us a sequence of smooth maps 
$F^m_k : M_m \rightarrow \B_{g_k} (x_k , m+1) \subset \m_k,$ for $k \geq m,$
mapping $x^m_{\infty}$ to $x_k,$ diffeomorphic onto their images 
and such that 
$ ( F^m_k )^{\ast} g_k (t) \rightarrow \hat{g}_m(t)$ 
smoothly uniformly on $M_m \times \left[ \de, T_m \right]$
as $k\to\infty$,
for every $\de \in (0,T_m)$.
Finally, we have 
\beq
\label{rad_2_a}
\twopartcond 
	{\Ric_{\hat{g}_m(t)} \geq - \al_m } { M_m \times \left( 0 ,T_m \right]}
	{| \Rm|_{\hat{g}_m(t)} \leq \frac{C_m}{t}} { M_m \times \left( 0 ,T_m \right].}
\eeq
By taking an appropriate  diagonal subsequence in $k$, we can be sure that these limits exist for \emph{every} $m\in\N$.

We now wish to relate the limit flows $\hat{g}_m(t)$ that we have 
constructed, for different $m$. Let us fix $m$. Then $\hat g_m(T_{m+1})$ is a smooth limit of the metrics $g_k(T_{m+1})$ 
(modulo the diffeomorphisms $F_k^{m}$) defined on $M_m.$
On the other hand, $\hat{g}_{m+1} ( T_{m+1})$ is a smooth limit of the metrics $g_k (T_{m+1})$ 
(modulo the diffeomorphisms $F_k^{m+1}$) defined on $M_{m+1}.$
Intuitively, $M_{m+1}$ should be ``bigger" than $M_m$ since it arises from the compactness of the metrics on larger radius balls.
This intuition is made precise in the following claim.

\vskip 8 pt

\begin{claim}
For sufficiently large $k$ we have 
\beq 
	\label{size_intuition_realised}
		F^m_k (M_m) \subset F^{m+1}_k ( M_{m+1} ).
\eeq
Indeed, we have the stronger inclusion that for any $t \in (0, T_{m+1}]$ and sufficiently large $k$, depending on $t$, 
\beq
	\label{more_precise_size_intuition_realised}
		F^m_k (M_m) \subset F^{m+1}_k \left( \B_{\hat{g}_{m+1} (t)} \left( x^{m+1}_{\infty} , m + 2 - 2\eta \right) \right) 
\eeq
which immediately yields \eqref{size_intuition_realised} via \eqref{size_of_M_m} by fixing $t=T_{m+1}$.
\end{claim}

\begin{claimproof}
Recall that by definition of $F^m_k$, for all $k\geq m\in \N$ we have 
$$F^m_k(M_m)\subset \B_{g_k}(x_k,m+1).$$
For each $t \in (0, T_{m+1}]$, and sufficiently large $k$, depending on $t$, we may appeal to Part 2 of Lemma \ref{dist funcs under loc conv},
with $ 2r = m + 2 - 2 \eta,$ $b = 2r,$ $a = m +2 - 3\eta,$ $x_0 = x^{m+1}_{\infty},$ $ (\n , \hat{g} ) = ( M_{m+1} , \hat{g}_{m+1} ( t ))$ 
and the sequence $ \{ \vph_i \}$ being the sequence $ \{ F^{m+1}_k \}_{k \geq m+1 },$
to deduce that $F^{m+1}_k ( \B_{\hat{g}_{m+1} ( t )} ( x^{m+1}_{\infty} , m+2 - 2\eta ) ) \supset \B_{g_k (t)} ( x_k , m + 2 - 3 \eta ).$
Thus, in order to prove \eqref{more_precise_size_intuition_realised}, it suffices to prove that
\beq
\label{STP_inclusion}
\B_{g_k}(x_k,m+1)\subset \B_{g_k (t)} ( x_k , m + 2 - 3 \eta ).
\eeq
We prove this through a combination of the shrinking and expanding balls lemmas. 

Recall from \eqref{compactness_reqs_a} we know that $\Ric_{g_k (t)} \geq -\al_m$ throughout $\B_{g_k} (x_k , m+2) \times [0,T_m]$
and $ | \Rm |_{g_k(t)} \leq \frac{C_m}{t} $ throughout $\B_{g_k} (x_k , m+2) \times (0,T_m].$
Therefore we can appeal to the shrinking balls lemma \ref{nested balls} to deduce that $ \B_{g_k ( t )} (x_k , m + 2 - 3\eta) \subset \B_{g_k} (x_k , m+2)$
provided $ m + 2 - 3\eta \leq m + 2 - \beta \sqrt{ C_m t},$
which will be the case if $ \beta \sqrt{C_m T_{m+1}} \leq 3 \eta,$ since $t \leq T_{m+1}.$
But $(i)$ in \eqref{dist_est_reduce_a} tells us that $\beta \sqrt{C_m T_m} < \frac{1}{8},$ which is slightly stronger than required (recalling the monotonicity of the sequence $T_j$).

Thus we may conclude that $\Ric_{g_k ( t )} \geq - \al_m$ throughout  $ \B_{g_k ( t )} (x_k , m + 2 - 3\eta) \times [0, T_{m+1}].$
The expanding balls lemma \ref{expanding balls} then tells us that $ \B_{g_k (t)} ( x_k , m + 1 + \eta ) \supset \B_{g_k} (x_k , m+1),$
provided $ (m + 1 + \eta ) e^{-\al_m t} \geq m +1,$
which will itself be true if $ (m+1) ( e^{\al_m T_{m+1}} - 1 ) \leq \eta.$
Since $T_m \geq T_{m+1}$ we observe that $(ii)$ of \eqref{dist_est_reduce_a} ensures this is the case.
But this inclusion is stronger than the inclusion
\eqref{STP_inclusion} that we need.
\end{claimproof}

\vskip 8pt
\noindent
By the uniqueness of smooth limits the metrics must agree in the sense that there is a smooth map $\psi_m : M_m \to M_{m+1}$
that is an isometry when domain and target are given the metrics $\hat{g}_m (T_{m+1})$ and $\hat{g}_{m+1} (T_{m+1})$ respectively, 
and which sends $x^m_{\infty}$ to $x^{m+1}_{\infty}.$

Indeed, after passing to another subsequence, we could see $\psi_m$ as a smooth limit, as $k \to \infty,$ of maps $ (F^{m+1}_k)^{-1} \circ F^m_k$, which are well-defined because of the claim, and 
which are independent of time, and it is apparent that in fact $\psi_m$ is an isometry also when domain and target are given the metrics $\hat{g}_m (t)$
and $\hat{g}_{m+1} (t)$ respectively, for any $t \in (0, T_{m+1}].$

Moreover seeing $\psi_m$ as such a limit and appealing to \eqref{more_precise_size_intuition_realised} allows us to conclude that 
\beq
	\label{helps_for_exhaustion}
		\psi_m (M_m) \subset \B_{\hat{g}_{m+1} (t)} \left( x^{m+1}_{\infty} , m + 2 - 2\eta \right)
\eeq
for any $t \in (0, T_{m+1}].$

At this point we can already define a smooth extension of $\hat{g}_{m+1}(t)$ to the longer time interval $t \in (0 , T_m],$
albeit on the smaller region $\psi_m (M_m),$
by taking $ ( \psi_m^{-1} )^{\ast} (\hat{g}_m (t)).$
However we would like to make such an extension for each $m,$ and we must pause to construct the manifold on which this final flow will live.

The maps $\psi_m : M_m \to M_{m+1}$ allow us to apply Theorem \ref{smooth manifold construction} to the collection $\{ M_m \}_{m \in \N}$.
Doing so gives a smooth three-manifold $M,$ a point $x_0 \in M,$ and smooth maps $\theta_m : M_m \to M,$ mapping $x^m_{\infty}$ to $x_0,$
diffeomorphic onto their images, satisfying $\theta_m (M_m) \subset \theta_{m+1} ( M_{m+1})$ and $\theta^{-1}_{m+1} \circ \theta_m = \psi_m,$
and such that 
\beq
	\label{limit_manifold_decomp}
		M = \bigcup_{m \in \N} \theta_m (M_m).
\eeq
In a moment, we will strengthen the inclusion $\th_m ( M_m ) \subset \th_{m+1} ( M_{m+1})$ to assert that the images of $M_m$ are contained within bounded subsets of $M.$

We can thus consider pull-back Ricci flows $(\th_m^{-1})^*\hat g_m(t)$ on $\th_m(M_m)\subset M$ for each $m$, and because $\psi_m$ is an isometry, these pull-backs agree where they overlap.
The union of the pull-backs we call $g(t)$.
Moreover, the curvature estimates of \eqref{rad_2_a} immediately give that for each $m \in \N$ we have 
\beq
\label{limit_flow_curv_ests_1_a}
	\twopartcond
		{\Ric_{g(t)} \geq -\al_m} { \theta_m ( M_m ) \times \left(0, T_m \right]}
		{|\Rm|_{g(t)} \leq \frac{C_m}{t}} {\theta_m ( M_m ) \times \left(0, T_m \right].}
\eeq
Furthermore, from \eqref{size_of_M_m} and \eqref{M_m_compactly_contained} we have that
\beq
	\label{size_of_image_M_m}
		\B_{g(s)} (x_0 , m + 1 - 3\eta ) = \theta_m ( \B_{\hat{g}_m (s)} ( x^m_{\infty} , m + 1 - 3 \eta ) ) \subset \subset \theta_m (M_m)
\eeq
for any $0 < s \leq T_m.$

Since $\th_{m+1}^{-1} \circ \th_m \equiv \psi_m,$ \eqref{helps_for_exhaustion} implies
$\th_{m+1}^{-1} \left( \th_m (M_m) \right) \subset \B_{\hat{g}_{m+1} (t)} \left( x^{m+1}_{\infty} , m + 2 - 2\eta \right) \subset M_{m+1}$
for any $t \in (0, T_{m+1}].$
Therefore we can strengthen the inclusion $\th_m ( M_m ) \subset \th_{m+1} ( M_{m+1})$ to 
\beq
	\label{M_m_image_not_too_crazy}
		\th_m (M_m) \subset \B_{g(t)} (x_0 , m + 2 - 2\eta) \subset \th_{m+1} ( M_{m+1})
\eeq
for any $t \in (0, T_{m+1}].$

For each $m \in \N$ we have a sequence 
$$f^m_k : \theta_m (M_m) \rightarrow \B_{g_k} (x_k , m+1) \subset \m_k$$
of smooth maps, for $k \geq m,$ 
defined by $f^m_k := F^m_k \circ \theta^{-1}_m,$ that map $x_0$ to $x_k$ and are diffeomorphic onto their images.
Moreover, from the choice of our diagonal subsequence, for any 
$\de \in (0,T_m)$ we have
\beq
\label{pullback_converge_limit_flow_a}
	\left( f^m_k \right)^{\ast} g_k (t) \rightarrow g(t)
\eeq
smoothly uniformly on $\theta_m (M_m) \times [ \de , T_m]$
as $k \rightarrow \infty.$

The obvious idea for constructing a distance metric $d$ on $M$ is to define $d := \lim_{t \downarrow 0} d_{g(t)}$, if we can be sure
that this limit exists.
The existence is a consequence of Lemma \ref{bi-holder distance estimates}, 
which may be applied with 
$r = \frac{m}{2} + \frac{1}{4},$ $\al=\al_m$, $c_0=C_m$ and $T=T_m$,
which is possible due to the curvature estimates of \eqref{limit_flow_curv_ests_1_a},
and the fact that from \eqref{size_of_image_M_m} we have 
$ \B_{g(s)} (x_0 , m + 1 - 3 \eta ) \subset \subset \theta_m (M_m)$
for any $0 < s \leq T_m.$

The result is a distance metric $d$ on $\Si_m := \bigcap_{t \in \left(0, T_m \right]} \B_{g( t)} \left( x_0 , \frac{m}{2} + \frac{1}{4} \right)$
arising as the  uniform limit of $d_{g(t)}$ as $t \downarrow 0.$ 
Moreover, 
for any $x,y \in \Si_m$ and any $0 < s \leq T_m$ we have
\beq 
	\label{limit_lip_est_m}
		d(x,y) - \beta \sqrt{ C_m s} \leq d_{g(s)} (x,y) \leq e^{\al_m s} d(x,y)
\eeq
and
\beq
	\label{limit_hol_est_m}
		\kappa_m ( m , \al_0 , v_0 ) \left[ d(x,y) \right]^{1+4C_m} \leq d_{g(s)} (x,y),
\eeq
where 
$\kappa_m > 0.$
As stated in Lemma \ref{bi-holder distance estimates}, these estimates ensure $d$ generates the same topology as we already have on $\Si_m.$ 

If we can estimate the $R_0$ from \eqref{conc 3} 
by $R_0>\frac{m}{2}+ \frac{1}{8}$, 
then \eqref{conc 3} gives that
for any $t \in (0,T_m]$ we have
\beq
\label{size_of_d_def_a}
	\B_d \left( x_0 , \frac{m}{2} + \frac{1}{8} \right) \subset \subset \cO_m \quad \text{and} \quad \B_{g(t)} \left( x_0 , \frac{m}{2} + \frac{1}{8} \right) \subset \subset \Si_m
\eeq
where $\cO_m$ is the connected component of the interior of $\Si_m$ that contains $x_0.$
This lower bound for $R_0$ is true provided $ \left( \frac{m}{2} + \frac{1}{4} \right)e^{-\al_m T_m} - \beta \sqrt{ C_m T_m} > \frac{m}{2} + \frac{1}{8},$
i.e. if $ 1 - 8 \beta \sqrt{C_m T_m} > ( 4m + 2 ) ( 1 - e^{-\al_m T_m}).$
Restriction $(i)$ in \eqref{dist_est_reduce_a} implies this inequality and hence the inclusions of \eqref{size_of_d_def_a} are valid.

A particular consequence of the first of these inclusions, via 
\eqref{limit_lip_est_m} and \eqref{limit_hol_est_m}, is that 
for any $x,y \in \B_d(x_0,\frac{m}{2})$ and any $0 < s \leq T_m$ we have
\beq 
	\label{limit_lip_est_m2}
		d(x,y) - \beta \sqrt{ C_m s} \leq d_{g(s)} (x,y) \leq e^{\al_m s} d(x,y)
\eeq
and
\beq
	\label{limit_hol_est_m2}
		\kappa_m ( m , \al_0 , v_0 ) \left[ d(x,y) \right]^{1+4C_m} \leq d_{g(s)} (x,y).
\eeq

The natural idea for extending $d$ to the entirety of $M$ is to repeat this procedure for all $m \in \N.$ 
Of course this will require the sets $ \{ \Si_m \}_{ m \in \N }$ to exhaust $M.$
That this is indeed the case is a consequence of the following claim.

\begin{claim}
For every $m \in \N$ we have $\th_m ( M_m) \subset \subset \Si_{2m+4}.$
\end{claim}
\vskip 4 pt
\begin{claimproof}
Recall from \eqref{M_m_image_not_too_crazy} we know that $\th_m ( M_m ) \subset \B_{g(t)} (x_0 , m + 2 - 2\eta)$ for any $t \in (0, T_{m+1}].$
Moreover \eqref{size_of_d_def_a} gives that for any $t \in (0, T_{2m+4}]$ we have $\B_{g(t)} (x_0 , m + 2 ) \subset \subset \Si_{2m+4}.$
Working with $t = T_{2m+4}$ in both of these inclusions gives the desired inclusion.
\end{claimproof}

\vskip 8 pt
\noindent
Knowing that the collection $\{ \Si_m \}_{m \in \N }$ exhausts $M$ allows us to repeat for all $m \in \N$ 
and extend $d$ to the entirety of $M$ whilst ensuring $d$ generates the same topology as we already have on $M.$

Moreover, it is clear that $(M,d)$ is a complete metric space. To elaborate, consider a Cauchy sequence in $M$ with respect to $d.$
This sequence must be bounded and so contained within $\B_d \left( x_0 , \frac{m}{2} \right)$ for some $m \in \N.$ 
The first inclusion of \eqref{size_of_d_def_a} tells us that the closure of this ball is compact, so we may pass to a convergent subsequence. 
By virtue of the sequence being Cauchy, this establishes the sequence itself is convergent.

The estimates \eqref{limit_lip_est_m2} and \eqref{limit_hol_est_m2} give the local bi-H\"{o}lder regularity of the identity map on $M$
that is claimed at the end of Theorem \ref{overarching_thm}, as we now explain.
Let $m \in \N$ and consider $\B_d \left(x_0 , \frac{m}{2} \right) \subset \subset M.$ 
For our arbitrary complete metric $g$ on $M$, 
the distance metric $d_g$ is bi-Lipschitz equivalent to $d_{g (T_m)}$ once restricted to $\B_d \left(x_0 , \frac{m}{2}\right).$ 
The estimates \eqref{limit_lip_est_m2} and \eqref{limit_hol_est_m2} tell us that the identity map
$ \left( \B_d \left(x_0 , \frac{m}{2} \right) , d \right) \to \left( \B_d \left(x_0 , \frac{m}{2} \right) , d_{g(T_m)} \right)$
is Lipschitz continuous, whilst  the identity map
$\left( \B_d \left(x_0 , \frac{m}{2} \right) , d_{g(T_m)} \right) \to \left( \B_d \left(x_0 , \frac{m}{2} \right) , d \right)$
is H\"{o}lder continuous, with Lipschitz constant and H\"{o}lder exponent depending only on $\al_0,$ $v_0$ and $m.$
The arbitrariness of $m \in \N$ gives the desired local bi-H\"{o}lder regularity of the identity map $(M,d) \to (M,d_g).$  


Having $d$ defined globally on $M$ allows us to simplify several of the techniques utilised in \cite{Topping2}. For example, given $m \in \N$ 
the local uniform convergence of $d_{g(t)}$ to $d$ as $t \downarrow 0$ tells us that for some $t_0 > 0$ we have 
$ \B_d (x_0 , m) \subset \B_{g(t)} \left(x_0 , m + \frac{1}{2} \right)$
for every $t \in (0 , \min \{ t_0 , T_m \} ].$
Hence from \eqref{size_of_image_M_m} (recalling the definition of $\eta$)
\beq
\label{ball_i_want_a}
	\B_d (x_0 , m) \subset \subset \theta_m (M_m)
\eeq
and so 
the estimates of \eqref{limit_flow_curv_ests_1_a}
are valid on $\B_d (x_0 , m ) \times \left( 0 , T_m \right].$
In fact, this establishes that the flow $g(t)$ lives where specified by the theorem.

We now turn our attention to defining the smooth maps $\vph_i.$
For each $m \in \N$,
by \eqref{size_of_image_M_m} and \eqref{pullback_converge_limit_flow_a}
we have $\left( f^m_k \right)^{\ast} g_k (T_m) \rightarrow g (T_m)$ smoothly on $\overline{\B_{g(T_m)} \left( x_0 , m + 1 - 4\eta \right)}$ 
and so, by appealing to
Lemma \ref{dist funcs under loc conv}, we may choose $K(m)$ such that 
for all $k\geq K(m)$ we have
\beq
\label{uniform_difference_dist_est_a}
\textstyle
	\left| d_{g_{k}(T_m)} \left( f^m_{k} (x) , f^m_{k} (y) \right) - d_{g(T_m)} (x,y) \right| \leq \frac{1}{m},
\eeq
\beq 
\label{uniform_quotient_dist_est_a}
\textstyle
	\left( 1 + \frac{1}{m} \right)^{-1} d_{g(T_m)} (x,y) \leq d_{g_{k}(T_m)} \left( f^m_{k} (x) , f^m_{k} (y) \right) \leq \left( 1 + \frac{1}{m} \right) d_{g(T_m)} (x,y) 
\eeq 
for all $x,y \in \B_{g(T_m)} \left( x_0 , \frac{m}{2} + \frac{1}{4} \right),$
and 
\beq
\label{smooth_converge_inclusion_a}
\textstyle
	f^m_k \left( \B_{g (T_m)} \left( x_0 , \frac{m}{2} - \frac{1}{2}\right) \right) \supset \supset \B_{g_k (T_m)} \left( x_k , \frac{m}{2} - \frac{3}{4}\right),
\eeq
where \eqref{smooth_converge_inclusion_a} will be required later to ensure the image of the (not yet defined) map $\vph_i$ is large enough.
We may assume that $K(m)$ is strictly increasing in $m$, otherwise we can fix $K(1)$, and then inductively replace $K(m)$ for $m=2,3,\ldots$ by the maximum of $K(m)$ and $K(m-1)+1$.
Pass to a further subsequence in $k$ by selecting the entries $K(1), K(2), K(3), \ldots $, so estimates \eqref{uniform_difference_dist_est_a}, \eqref{uniform_quotient_dist_est_a} and \eqref{smooth_converge_inclusion_a}
now hold for all $k\geq m$.

For each $i \in \N$ we define a map $\vph_i : \theta_i (M_i) \to \B_{g_i} ( x_{i} , i + 1 ) \subset \m_i$ 
by $\vph_i := f^i_i.$
In particular, each $\vph_i$ is defined throughout $\B_d (x_0 , i)$ thanks to \eqref{ball_i_want_a}.
These are diffeomorphisms onto their images, map $x_0$ to $x_i$ and satisfy versions of the above estimates. Namely
\beq
\label{vph_uniform_difference_dist_est_a}
\textstyle
	\left| d_{g_{i}(T_i)} \left( \vph_i (x) , \vph_i (y) \right) - d_{g(T_i)} (x,y) \right| \leq \frac{1}{i},
\eeq
\beq 
\label{vph_uniform_quotient_dist_est_a}
\textstyle
	\left( 1 + \frac{1}{i} \right)^{-1} d_{g(T_i)} (x,y) \leq d_{g_{i}(T_i)} \left( \vph_i (x) , \vph_i (y) \right) \leq \left( 1 + \frac{1}{i} \right) d_{g(T_i)} (x,y) 
\eeq 
for all $x,y \in \B_{g(T_i)} \left( x_0 , \frac{i}{2} + \frac{1}{4} \right),$
and 
\beq
\label{vph_smooth_converge_inclusion_a}
\textstyle
	\vph_i \left( \B_{g (T_i)} \left( x_0 , \frac{i}{2} - \frac{1}{2}\right) \right) \supset \supset \B_{g_i (T_i)} \left( x_i , \frac{i}{2} - \frac{3}{4}\right).
\eeq
In what follows we will fix some $i_0 \in \N$ and consider the maps $\vph_i$ for $i \geq i_0$ restricted to the ball $\B_d (x_0 , i_0).$ 
With this in mind we record the following observations.

Given a fixed $i_0 \in \N,$ restriction $(ii)$ in \eqref{dist_est_reduce_a} (recalling the definition of $\eta$) ensures that $i_0 e^{\al_{i_0} T_{i_0}} < i_0 + \frac{1}{2}.$
Hence \eqref{limit_lip_est_m2} and the monotonicity of the sequence $T_i$ imply that for all $i \geq i_0$ we have the inclusion
\beq
	\label{d-ball_later_inc}
		\B_d \left(x_0 , \frac{i_0}{2} \right) \subset \B_{g(T_i)} \left( x_0 , \frac{i_0}{2} + \frac{1}{4} \right).
\eeq
This inclusion implies that for $i \geq i_0$ both \eqref{vph_uniform_difference_dist_est_a} and \eqref{vph_uniform_quotient_dist_est_a} are valid for all $x,y \in \B_d \left( x_0 , \frac{i_0}{2} \right).$
Moreover, restriction $(i)$ in \eqref{dist_est_reduce_a}  ensures
that $ \beta \sqrt{C_i T_i} <\frac{1}{8},$ and so \eqref{limit_lip_est_m} and \eqref{size_of_d_def_a} (with $i$ here being  used as the $m$ there) yields that
$\B_d \left( x_0 , \frac{i}{2} \right) \supset \B_{g (T_i)} \left( x_0 , \frac{i}{2} - \frac{1}{2} \right).$
Hence \eqref{vph_smooth_converge_inclusion_a} implies  $\vph_i \left( \B_d \left( x_0 , \frac{i}{2} \right) \right) \supset \B_{g_i (T_i)} \left( x_i , \frac{i}{2} - \frac{3}{4}\right).$

Now we restrict $\vph_i$ to the ball $\B_d (x_0 , i).$ 
Above we have shown that for any $i \in \N$ we have
\beq
\label{vph_C}
\textstyle
	\vph_i \left( \B_d \left( x_0 , \frac{i}{2}\right) \right) \supset \supset \B_{g_i (T_i)} \left( x_i , \frac{i}{2} - \frac{3}{4}\right).
\eeq
Moreover, given $i_0 \in \N$ we have shown that for all $ i \geq i_0$ we have
\beq
\label{vph_A}
\textstyle
	\left| d_{g_{i}(T_i)} \left( \vph_i (x) , \vph_i (y) \right) - d_{g(T_i)} (x,y) \right| \leq \frac{1}{i},
\eeq
\beq 
\label{vph_B}
\textstyle
	\left( 1 + \frac{1}{i} \right)^{-1} d_{g(T_i)} (x,y) \leq d_{g_i(T_i)} \left( \vph_i (x) , \vph_i (y) \right) \leq \left( 1 + \frac{1}{i} \right) d_{g(T_i)} (x,y) 
\eeq 
for all $x,y \in \B_d \left( x_0 , \frac{i_0}{2}  \right).$

We now turn our attention to the properties of these maps restricted to balls of the form $\B_d (x_0 , R).$
We first establish the uniform convergence and bi-H\"{o}lder regularity claims. 
For this purpose we take $i_0$ to be $i_0 := 2 \left( \lfloor R \rfloor + 1 \right) \in \N.$

For $i \geq i_0$ the \emph{pyramid Ricci flow} $g_i (t)$ is defined on $\cd_i$ (recall \eqref{domain_of_def_g_k_a}), 
and in particular \eqref{compactness_reqs_a} gives that $\Ric_{g_i(t)} \geq -\al_{i_0}$ throughout $\B_{g_i} ( x_i , i_0 + 2) \times [ 0 , T_{i_0} ]$
and $|\Rm |_{g_i (t)} \leq \frac{C_{i_0}}{t}$ throughout $\B_{g_i } ( x_i , i_0 + 2) \times ( 0 , T_{i_0} ].$
But restriction $(i)$ of \eqref{dist_est_reduce_a} tells us that $ \beta \sqrt{ C_{i_0} T_{i_0}} <\frac{1}{8}$, so 
the shrinking balls lemma \ref{nested balls} gives that 
$$ \textstyle
\B_{g_i (s) } \left( x_i , i_0 + 2 - \frac{1}{8} \right) \subset \B_{g_i} \left( x_i , i_0 + 2 - \frac{1}{8} + \beta \sqrt{C_{i_0} T_{i_0}} \right) \subset \subset \B_{g_i} ( x_i , i_0 + 2)$$
for any $s \in [0, T_{i_0} ].$
These estimates
allow us to apply Lemma \ref{bi-holder distance estimates} to the flow $g_i (t)$ with 
$r = \frac{i_0}{2} + 1 - \frac{1}{16},$ $n=3,$ $\al = \al_{i_0},$ $c_0 = C_{i_0}$ and $T = T_{i_0}$ 
to quantify the uniform convergence of $d_{g_i(s)}$ 
to $d_{g_i}$ as $s \downarrow 0$
on $ \Omega^{i_0}_i := \bigcap_{0 < t \leq T_{i_0}} \B_{g_i (t)} \left( x_i , \frac{i_0}{2} + 1 - \frac{1}{16} \right)$.
%
For any $z,w \in \Omega^{i_0}_i$ and any $0 < s \leq T_{i_0}$ we have
\beq
	\label{lip_est_seq_a}
		d_{g_i} (z,w) - \beta \sqrt{ C_{i_0} s} \leq d_{g_i (s)} (z,w) \leq e^{\al_{i_0} s} d_{g_i} (z,w)
\eeq
and
\beq
	\label{hol_est_seq_a}
		\gamma ( i_0 , \al_0 , v_0 ) \left[ d_{g_i} (z,w) \right]^{1+4C_{i_0}} \leq d_{g_i (s)} (z,w),
\eeq
where 
$ \gamma > 0.$

If we can estimate the $R_0$ from \eqref{conc 3} 
by $R_0> \frac{i_0}{2} + \frac{1}{2}$, 
then \eqref{conc 3} gives that
\beq
\label{size_of_seq_ests_a}
	\B_{g_i (s)} \left( x_i , \frac{i_0}{2} + \frac{1}{2} \right) \subset \subset \Omega^{i_0}_i
\eeq
for any $ 0 \leq s \leq T_{i_0},$ recalling that $g_i(0) = g_i$ on $\B_{g_i} ( x_i , i + 2).$
This lower bound for $R_0$ will be true provided $ \left( \frac{i_0}{2} + 1 - \frac{1}{16} \right) e^{- \al_{i_0} T_{i_0} } - \beta \sqrt{ C_{i_0} T_{i_0}} > \frac{i_0}{2} + \half.$
This inequality is itself true if $ \frac{7}{2} - 8 \beta \sqrt{ C_{i_0} T_{i_0} } > \left( 4 i_0 + 8 - \frac{1}{2} \right) ( 1 - e^{-\al_{i_0} T_{i_0}} ).$
Restriction $(i)$ in \eqref{dist_est_reduce_a} implies this latter inequality, and so the inclusions of \eqref{size_of_seq_ests_a} are valid.

We are now ready to establish the claimed uniform convergence. To do so we closely follow the argument of Miles Simon and the second author utilised in the proof of Theorem 1.4 in \cite{Topping2}.

\begin{claim}
As $i \to \infty,$ we have convergence 
\beq
	\label{uniform_converge_claim}
		d_{g_i} ( \vph_i (x) , \vph_i(y)) \to d(x,y)
\eeq
uniformly as $x,y$ vary over $\B_d \left(x_0 , \frac{i_0}{2} \right).$
\end{claim}

\begin{claimproof}
Let $\ep > 0.$ We must make sure that for sufficiently large $i,$ depending on $\ep,$ we have
\beq
	\label{aiming_for_one}
		| d_{g_i} ( \vph_i (x) , \vph_i(y)) - d(x,y)| < \ep
\eeq
for all $x,y \in \B_d \left(x_0 , \frac{i_0}{2}\right).$
By the distance estimates \eqref{lip_est_seq_a} and the inclusions of \eqref{size_of_seq_ests_a} there exists a $\tau_1 >0,$ depending only on $\ep, i_0, \al_0$ and $v_0,$
such that for all $i \geq i_0$ and any $s \in (0 , \min \{ \tau_1 , T_{i_0} \}]$ we have 
\beq
	\label{part_1_a}
		| d_{g_i} (z,w) - d_{g_i (s)} (z,w) | < \frac{\ep}{3}
\eeq
whenever there exists $ t \in [ 0 , T_{i_0}]$ such that $z,w \in \B_{g_i (t)} \left( x_i , \frac{i_0}{2} + \frac{1}{2} \right).$

By the distance estimates \eqref{limit_lip_est_m2} (for $m = i_0$) 
there exists a $\tau_2 >0,$ depending only on $\ep, i_0, \al_0$ and $v_0,$
such that for any $s \in (0 , \min \{ \tau_2 , T_{i_0} \}]$ we have
\beq
	\label{part_2_a}
		| d(x,y) - d_{g(s)} (x,y) | < \frac{\ep}{3}
\eeq
for all $x,y \in \B_d \left(x_0 , \frac{i_0}{2} \right).$

Let $\tau := \min \{ \tau_1 , \tau_2 \} > 0$ (though we could have naturally picked the same $\tau_1$ and $\tau_2$ to begin with) and choose $i_1 \in \N$ such that for all $i \geq i_1$ we have $T_i < \tau;$ 
this is possible since $T_i \downarrow 0$ as $i \to \infty.$ Therefore for $i \geq \max \{ i_0 , i_1 \}$ both \eqref{part_1_a} and \eqref{part_2_a}
hold for $s = T_i.$

From \eqref{vph_A}, for all $i \geq \max \left\{ i_0 , \frac{3}{\ep} \right\}$ we have 
\beq
	\label{part_3_a}
		\left| d_{g_i (T_i)} \left( \vph_i (x) , \vph_i (y) \right) - d_{g(T_i)} (x,y) \right| < \frac{1}{i} < \frac{\ep}{3}
\eeq
for all $x,y \in \B_d \left( x_0 , \frac{i_0}{2} \right).$

Let $x,y \in \B_d \left( x_0 , \frac{i_0}{2} \right)$ and let $i \geq \max \left\{ i_0 , i_1 , \frac{3}{\ep}, 5 \right\}.$
Appealing to \eqref{d-ball_later_inc} gives $x , y \in \B_{g(T_i)} \left( x_0 , \frac{i_0}{2} + \frac{1}{4} \right),$ thus
\eqref{part_3_a} tells us that $ \vph_i (x) , \vph_i(y) \in \B_{g_i (T_i)} \left( x_i , \frac{i_0}{2} + \frac{1}{4} + \frac{1}{i} \right).$
Since $ i \geq 5$ this tells us that \eqref{part_1_a} is valid for $z = \vph_i (x)$ and $w=\vph_i(y).$
Combining \eqref{part_1_a}, \eqref{part_2_a} and \eqref{part_3_a} establishes \eqref{aiming_for_one} and completes the proof of the claim.
\end{claimproof}

\vskip 8pt
\noindent
The uniform convergence on $\B_d \left( x_0 , \frac{i_0}{2} \right)$ immediately gives uniform convergence on $\B_d (x_0 , R)$ since $\frac{i_0}{2} \geq R.$

The bi-H\"{o}lder estimates for $\vph_i |_{\B_d (x_0 , R)}$ are an easy consequence of those we have already obtained. 
If $x,y \in  \B_d \left( x_0 , \frac{i_0}{2} \right)$ then for $i \geq i_0$ \eqref{d-ball_later_inc} yields that
$x,y \in \B_{g(T_i)} \left( x_0 , \frac{i_0}{2} + \frac{1}{4} \right).$
Then \eqref{vph_A} gives $\vph_i (x) , \vph_i (y) \in \B_{g_i (T_i)} \left( x_i , \frac{i_0}{2} + \frac{1}{4} + \frac{1}{i} \right).$
Thus for $i \geq \max \{ i_0 , 5 \}$ we have $\vph_i (x) , \vph_i (y) \in \B_{g_i (T_i)} \left( x_i , \frac{i_0}{2} + \frac{1}{2} \right).$
Therefore by \eqref{size_of_seq_ests_a} both the estimates \eqref{lip_est_seq_a} and \eqref{hol_est_seq_a} are valid for $z = \vph_i(x)$ and $w = \vph_i (y).$

As a first consequence, we can deduce that for all $x,y \in \B_d \left( x_0 , \frac{i_0}{2}\right)$ and all $i \geq \max \{ i_0 , 5\}$ we have
\begin{align*}
	d(x,y) &\stackrel{\eqref{limit_hol_est_m2}}{\leq} \left[ \frac{1}{\kappa_{i_0} (i_0 , \al_0 , v_0)} d_{g (T_i)} ( x , y) \right]^{\frac{1}{1+4C_{i_0}}} \\
		&\stackrel{\eqref{vph_B}}{\leq} \left[ \frac{\left(1+\frac{1}{i}\right)}{\kappa_{i_0} (i_0 , \al_0 , v_0)} d_{g_i (T_i)} ( \vph_i (x) , \vph_i (y) ) \right]^{\frac{1}{1+4C_{i_0}}}  \\
		&\stackrel{\eqref{lip_est_seq_a}}{\leq} \left[ \frac{\left(1+\frac{1}{i}\right)e^{\al_{i_0} T_i } }{\kappa_{i_0} (i_0 , \al_0 , v_0)} d_{g_i } ( \vph_i (x) , \vph_i (y) ) \right]^{\frac{1}{1+4C_{i_0}}}.
\end{align*}
The monotonicity of the sequence $T_i$ allows us to define 
$ B (i_0 , \al_0 , v_0 ) := \left[ \frac{2e^{\al_{i_0} T_{i_0} } }{\kappa_{i_0} (i_0 , \al_0 , v_0)} \right]^{\frac{1}{1+4C_{i_0}}} > 0$
and conclude that for all $i \geq \max \{ i_0 , 5 \}$ we have
\beq
	\label{bi_hol_part_1}
		d(x,y) \leq B (i_0 , \al_0 , v_0 ) \left[  d_{g_i } ( \vph_i (x) , \vph_i (y) ) \right]^{\frac{1}{1+4C_{i_0}}}.
\eeq
Similarly, a second consequence is that for all $x,y \in \B_d \left( x_0 , \frac{i_0}{2} \right)$ and all $i \geq \max \{ i_0 , 5 \}$ we have
\begin{align*}
	d_{g_i}( \vph_i(x) , \vph_i (y) ) &\stackrel{\eqref{hol_est_seq_a}}{\leq} \left[ \frac{1}{\gamma (i_0 , \al_0 , v_0)} d_{g_i (T_i)} ( \vph_i (x) , \vph_i (y)) \right]^{\frac{1}{1+4C_{i_0}}} \\
		&\stackrel{\eqref{vph_B}}{\leq} \left[ \frac{\left(1+\frac{1}{i}\right)}{\gamma (i_0 , \al_0 , v_0)} d_{g (T_i)} ( x , y ) \right]^{\frac{1}{1+4C_{i_0}}}  \\
		&\stackrel{\eqref{limit_lip_est_m2}}{\leq} \left[ \frac{\left(1+\frac{1}{i}\right)e^{\al_{i_0} T_i } }{\gamma (i_0 , \al_0 , v_0)} d ( x , y ) \right]^{\frac{1}{1+4C_{i_0}}}.
\end{align*}
The monotonicity of the sequence $T_i$ allows us to define 
$ A (i_0 , \al_0 , v_0 ) := \left[ \frac{2e^{\al_{i_0} T_{i_0} } }{\gamma (i_0 , \al_0 , v_0)} \right] > 0$
and conclude that for all $i \geq \max \{ i_0 , 5\}$ we have
\beq
	\label{bi_hol_part_2}
		d_{g_i}( \vph_i(x) , \vph_i(y) ) \leq A (i_0 , \al_0 , v_0 )^{\frac{1}{1+4C_{i_0}}} \left[  d ( x , y ) \right]^{\frac{1}{1+4C_{i_0}}}.
\eeq
Combining \eqref{bi_hol_part_1} and \eqref{bi_hol_part_2} yields that for all $x,y \in \B_d \left( x_0 , \frac{i_0}{2} \right)$ and all $ i \geq \max \{ i_0 , 5\}$ 
\beq
	\label{vph_holder_est}
		\frac{1}{A( i_0 , \al_0 , v_0 )} \left[ d_{g_i } ( \vph_i (x) , \vph_i (y) ) \right]^{1+4C_{i_0}} \leq  d(x,y) \leq B( i_0 , \al_0 , v_0 ) \left[ d_{g_i } ( \vph_i (x) , \vph_i (y) ) \right]^{\frac{1}{1+4C_{i_0}}}.
\eeq
This establishes that for all $i \geq \max \{ i_0 , 5\}$ the restriction of
$\vph_i$ to $\B_d \left( x_0 , \frac{i_0}{2} \right)$ is bi-H\"{o}lder with H\"{o}lder exponent depending only on $i_0, \al_0$ and $v_0.$ 
Since $\frac{i_0}{2} \geq R$ and $i_0$ is determined by $R,$ we deduce from \eqref{vph_holder_est}
that, for all $i \geq \max \{ i_0 , 5 \},$ the restriction of $\vph_i$ to $\B_d (x_0 , R)$ is bi-H\"{o}lder with H\"{o}lder exponent depending only on $\al_0, v_0$ and $R$ as desired.

Next we turn our attention to the claim that the image of $\B_d (x_0 , R)$ under $\vph_i$ is eventually arbitrarily close to being the whole of $\B_{g_i} ( x_i , R ).$
We know $ \vph_i \left( \B_d \left( x_0 , \frac{i}{2} \right) \right) \supset \B_{g_i (T_i)} \left( x_i , \frac{i}{2} - \frac{3}{4} \right)$ from \eqref{vph_C}.
We claim that $\B_{g_i (T_i)} \left( x_i , \frac{i}{2} - \frac{3}{4} \right) \supset \B_{g_i} \left( x_i , \frac{i}{2} - 1 \right).$ 
To begin with we can appeal to the shrinking balls lemma \ref{nested balls} to deduce that
$$ \B_{g_i (T_i)} \left(x_i , \frac{i}{2} - \frac{3}{4} \right) \subset \B_{g_i } \left(x_i , \frac{i}{2} - \frac{3}{4} + \frac{1}{8} \right) \subset \subset \B_{g_i} ( x_i , i +2 )$$
since $1 - 8 \beta \sqrt{ C_i T_i } > 0.$ 
This inclusion implies that $\Ric_{g_i(t)} \geq -\al_i$ throughout $\B_{g_i (T_i)}  \left( x_i , \frac{i}{2} - \frac{3}{4} \right) \times [0,T_i].$
The expanding balls lemma \ref{expanding balls} now gives our desired inclusion provided 
$ \left( \frac{i}{2} - \frac{3}{4} \right) e^{-\al_i T_i} \geq \frac{i}{2}-1,$ that is if $ \left( i - \frac{3}{2} \right) ( 1 - e^{-\al_i T_i} ) \leq \frac{1}{2}.$
However this is guaranteed to be true by $(ii)$ in \eqref{dist_est_reduce_a}, which imposed the stronger condition $ (i+1) ( e^{\al_i T_i} - 1  ) \leq \eta.$
Therefore for all $i \geq 2(R + 1)$ we have that 
\beq
\label{weak_inclusion}
 \vph_i \left( \B_d (x_0 , i) \right) \supset 
\vph_i \left( \B_d (x_0 , \frac i2) \right) \supset
\B_{g_i} ( x_i , R).
\eeq
Now suppose $r\in (0,R)$ as in the theorem. By the uniform convergence claim \eqref{uniform_converge_claim}, we know that for sufficiently large $i$, let's say for $i\geq i_2$, we have
$|d_{g_i} ( \vph_i (x) , \vph_i(y)) - d(x,y)|<\frac{R-r}2$
for all $x,y\in \overline{\B_d(x_0,R)}$, and in particular,
\beq
\label{unif_cgnce_consequence}
d(x_0,y)<d_{g_i} ( x_i , \vph_i(y)) +\frac{R-r}2
\qquad\text{ for all }y\in \overline{\B_d(x_0,R)}.
\eeq
We claim that this implies our desired inclusion 
\beq
\label{des_inc}
\B_{g_i}(x_i,r)\subset \vph_i(\B_d(x_0,R))
\qquad\text{ for }i\geq i_2.
\eeq
If not, then, keeping in mind \eqref{weak_inclusion}, there exists $z\in \B_{g_i}(x_i,r)$ such that $y:=\vph_i^{-1}(z)\notin \B_d(x_0,R)$.
Because we have $d(x_0, y)>R$, we can move a point $\hat z$ along a minimising geodesic from $x_i$ to $z$ until the first time that 
$d(x_0, \vph_i^{-1}(\hat z))=R$, then replace $z$ by $\hat z$.
This guarantees that additionally we have
$d(x_0, y)=R$ and $y\in \overline{\B_d(x_0,R)}$.
But then by \eqref{unif_cgnce_consequence}  we have
\beqa
R &= \textstyle d(x_0, y)< d_{g_i} ( x_i , z) +\frac{R-r}2\\
& < \textstyle r+\frac{R-r}2\\
& < R,
\eeqa
a contradiction. Thus \eqref{des_inc} holds as desired.


Finally we observe that, for sufficiently large $i \in \N,$ slight modifications of the maps $\vph_i$
give $\ep$-Gromov-Hausdorff approximations $\B_d (x_0 , R)$ to $\B_{g_i} (x_i , R).$ 
Since $R > 0$ is arbitrary, we deduce that $ ( \m_i , d_{g_i} , x_i ) \rightarrow ( M , d , x_0)$
in the pointed Gromov-Hausdorff sense as $i \rightarrow \infty.$
\end{proof}

\bcmt{I used quite vague language `slight modifications' above because it's a bit messy and arbitrary. One way of doing it is to take each point in $\B_d (x_0 , R)$ that $\vph_i$ maps outside $\B_{g_i} (x_i , R)$, and instead map it to the nearest point within $\B_{g_i} (x_i , R)$. This kills the smoothness, of course.}

\appendix

\section{Appendix - Results from Simon-Topping papers}\label{appA}

Here we collect statements of the various results from \cite{Topping1} and \cite{Topping2} that we require. 
We first record a scaled variant of Lemma 4.1 in \cite{Topping2}, where we have weakened the required Ricci lower bound to $-\gamma$ rather than $-1.$ Lemma 4.1 in \cite{Topping2} corresponds to the $\gamma = 1$ case. 
The same statement is actually given as Lemma 2.1 in \cite{Topping1}, but with less good dependencies given for the curvature estimates achieved. 
The following result makes explicit ideas that are implicit in \cite{Topping1} and \cite{Topping2}.

\begin{lemma}[Variant of the local lemma 4.1 in \cite{Topping2}]
\label{local lemma 1}
Given any $v_0>0$, there exists $C_0=C_0(v_0)\geq 1$ such that the following is true.
Let $\left( M^3 , g(t) \right),$ for $0 \leq t \leq T,$ be a smooth Ricci flow such that for some fixed $x \in M$ 
we have $\B_{g(t)} ( x , 1) \subset \subset M$ for all $ 0 \leq t \leq T,$ and so that for any $0 < r \leq 1$,
$\VolBB_{g(0)} (x,r) \geq v_0 r^3 $ and
$\Ric_{g(t)} \geq -\gamma$ on $\B_{g(t)}(x,1)$ for some $\gamma > 0$ and all $0 \leq t \leq T.$
Then there exists 
$S = S\left(v_0 , \gamma \right) > 0$ 
such that for all 
$0 < t \leq \min \left\{ T , S \right\}$ 
we have both 
\begin{equation} | \Rm|_{g(t)} (x) \leq \frac{C_0}{t} \quad \text{and} \quad \inj_{g(t)}(x) \geq \sqrt{\frac{t}{C_0}.} \end{equation}
\end{lemma}

\begin{proof}
Without loss of generality we assume that $\gamma \geq 1;$ if $0 < \gamma < 1$ then we could replace $\gamma$ by $1$ 
since $\Ric_{g(t)} \geq - \gamma$ would give that $\Ric_{g(t)} \geq -1.$ 
Then consider the rescaled flow $g_p (t) := \gamma g \left( \frac{t}{\gamma} \right)$ for times 
$0 \leq t \leq \gamma T.$ 
We first observe that 
\beq\label{t=0 vol} 
	\VolBB_{g_p(0)}  \left( x , 1 \right) = \gamma^{\frac{3}{2}} \VolBB_{g(0)}  \left( x , \frac{1}{\sqrt{\gamma}} \right)  \geq \gamma^{\frac{3}{2}} \gamma^{-\frac{3}{2}} v_0 = v_0. 
\eeq
Moreover, for any $0 \leq t \leq \gamma T$ we have both 
\beq\label{ball compact still} 
	\B_{g_p(t)} ( x , 1 ) = \B_{g \left( \frac{t}{\gamma} \right) } \left( x , \frac{1}{\sqrt{\gamma}} \right) \subset \subset M 
\eeq
and for any $z \in \B_{g_p(t)} (x,1)$ 
\beq\label{Ric preserved} 
\Ric_{g_p(t)} (z) = \Ric_{\gamma g \left( \frac{t}{\gamma} \right)} (z) \geq -1 
\eeq 
since $z \in \B_{g \left( \frac{t}{\gamma} \right)} (x,1).$ 
Therefore, by combining \eqref{t=0 vol}, \eqref{ball compact still} and \eqref{Ric preserved} we have the hypotheses to be able to apply Lemma 4.1 from \cite{Topping2}. 
Doing so gives us constants $C_0 = C_0 (v_0) \geq 1$ and $S_0 = S_0(v_0) > 0$ such that for all $0 < t \leq \min \left\{ \gamma T , S_0 \right\}$ we have both 
\beq\label{zxa} 
	| \Rm|_{g_p(t)} (x) \leq \frac{C_0}{t} \quad \text{and} \quad \inj_{g_p(t)}(x) \geq \sqrt{\frac{t}{C_0}.} 
\eeq
Both the estimates in \eqref{zxa}  are preserved under rescaling back to the original flow $g(t).$ 
Then, by taking $S := \frac{S_0}{\gamma} > 0,$ which does indeed depend only on $v_0$ and $\gamma,$ we deduce \eqref{zxa} for the flow $g(t)$ itself and for all times $0 < t \leq \min \left\{ T , S \right\}.$
\end{proof}

\begin{lemma}[Double bootstrap lemma 4.2 in \cite{Topping2} or Lemma 9.1 in \cite{Topping1}]
\label{double bootstrap}
Let $\left( M^3 , g(t) \right)$ be a smooth Ricci flow, for $0 \leq t \leq T,$ such that for some $x \in M$ we have $\B_{g(0)} ( x , 2) \subset \subset M,$ and so that 
\begin{itemize}
	\item $| \Rm |_{g(t)} \leq \frac{c_0}{t}$ on $\B_{g(0)} (x,2) \times \left(0 , T \right]$ for some $c_0 \geq 1$ and
	\item $\Ric_{g(0)} \geq -\delta_0$ on $\B_{g(0)}(x,2)$ for some $\delta_0 > 0.$
\end{itemize}
Then there exists $S = S(c_0 , \delta_0 ) > 0$ such that for all $0 < t \leq \min \left\{ T , S \right\}$ we have 
\begin{equation} \Ric_{g(t)}(x) \geq -100\delta_0 c_0. \end{equation}
\end{lemma}

\begin{theorem}[Local existence theorem 1.6 in \cite{Topping2}]
\label{local existence}
Suppose $s_0 \geq 4.$ Suppose $\left( M^3 , g_0 \right)$ is a Riemannian manifold, $x_0 \in M, \mathbb{B}_{g_0} (x_0 , s_0 ) \subset \subset M$
and $\Ric_{g_0} \geq - \alpha_0$ on $\mathbb{B}_{g_0} (x_0 , s_0)$ 
and 
$\VolBB_{g_0}  (x , 1)\geq v_0 > 0$ for all $ x \in \B_{g_0} (x_0 , s_0 - 1).$ 
Then there exist constants $T = T(\alpha_0 , v_0) > 0, \alpha = \alpha ( \alpha_0 , v_0) > 0, c_0 = c_0 ( \alpha_0 , v_0 ) > 0$ 
and a Ricci flow $g(t)$ defined for $0 \leq t \leq T$ on $\B_{g_0} (x_0 , s_0 - 2),$ 
with $g(0)=g_0$ where defined, such that for all $0 < t \leq T$ we have 
\begin{itemize}
	\item $\Ric_{g(t)} \geq - \al$ on $\B_{g_0} (x_0 , s_0 -2)$ and 
	\item $| \Rm |_{g(t)} \leq \frac{c_0}{t}$ on $\B_{g_0} (x_0 , s_0 -2).$
\end{itemize}
\end{theorem}

The following is a variant of Lemma 2.3 in \cite{Topping1}. We replace the required compactness of a time $t$ ball by compactness of a time $0$ ball. 
Moreover, we now obtain volume estimates for unit balls within a later time $t$ ball, rather than just for a single fixed unit ball at later times $t.$ 
Again this makes explicit ideas implicitly used in both \cite{Topping1} and \cite{Topping2}.

\begin{lemma}[Variant of lower volume control lemma 2.3 in \cite{Topping1}]
\label{Volume 2}
Suppose that $\left( M^n , g(t) \right)$ is a smooth Ricci flow over the time interval $t \in \left[0,T\right)$ and that 
for some $R \geq 2$ we have that 
$ \B_{g(0)} ( x_0 , R) \subset \subset M$ for some $x_0 \in M.$ 
Moreover assume that  
\begin{itemize}
	\item $\Ric_{g(t)} \geq -K$ on $\B_{g(0)} (x_0 , R ),$ for some $K > 0$ and all $t \in \left[0,T\right),$
	\item $| \Rm |_{g(t)} \leq \frac{c_0}{t}$ on $\B_{g(0)} (x_0 , R ),$ for some $c_0 > 0$ and all $t \in \left(0,T\right),$
	\item $\VolBB_{g(0)}  ( x_0 , 1 )  \geq v_0 > 0.$
\end{itemize}
Then there exists $\varepsilon_R = \varepsilon_R \left( v_0 , K , R , n \right) > 0$ and $\hat{T} = \hat{T} \left( v_0 , c_0 , K , n , R \right) > 0$ 
such that for all 
$t \in \left[0,T\right) \cap \left[0, \hat{T} \right)$ 
we have 
$\B_{g(t)} ( x_0 , R - 1 ) \subset \B_{g(0)} ( x_0 , R ),$ 
and that for all 
$x \in \B_{g(t)} ( x_0 , R-2),$ 
we have 
$\VolBB_{g(t)} ( x , 1 ) \geq \varepsilon_R.$
\end{lemma}

\begin{proof}
Lemma \ref{nested balls} yields a $\beta = \beta (n) \geq 1$ for which $\B_{g(0)} ( x_0 , R) \supset \B_{g(t)} ( x_0 , R - \beta \sqrt{c_0 t} )$ for all $t \in [0,T).$ 
Therefore, 
for $ 0 \leq t \leq \min \left\{ T , \frac{1}{\beta^2 c_0} \right\}$ 
we have 
$\B_{g(t)} ( x_0 , R - 1 ) \subset \B_{g(0)} ( x_0 , R ),$ so $\B_{g(t)}( x_0 , R-1 ) \subset \subset M$ 
and the assumed curvature estimates hold on $\B_{g(t)} ( x_0 , R-1 )$ for all such times $t.$ 
Lemma 2.3 in \cite{Topping1} with $\gamma = 1$ yields $\varepsilon_0 = \varepsilon_0 \left( v_0 , K , n \right) > 0$ and $\tilde{T} = \tilde{T} \left( v_0 , c_0 , K , n \right) > 0$ such that 
$\VolBB_{g(t)} \left( x_0 , 1 \right)  \geq \varepsilon_0 > 0 $
for all times $0 \leq t \leq \min \left\{ T , \frac{1}{\beta^2 c_0} , \tilde{T} \right\}.$ 
Set $\hat{T} := \min \left\{ \tilde{T} , \frac{1}{\beta^2 c_0} \right\} > 0$, which depends only on 
$v_0,$ $K,$ $c_0,$ $n$ and $R.$ 
Given any $t \in [ 0 , \min \{ T , \hat{T} \} ],$ 
the Ricci lower bound 
$\Ric_{g(t)} \geq -K$ 
throughout 
$\B_{g(t)} \left( x_0 , R -1 \right)$ allows us, via Bishop-Gromov, to reduce 
$\varepsilon_0$ to a constant $ \varepsilon_R = \varepsilon_R \left( v_0 , K , n , R \right) > 0$ 
such that for all 
$x \in \B_{g(t)} \left( x_0 , R-2 \right),$ 
we have 
$\VolBB_{g(t)}  (x,1)  \geq \varepsilon_R > 0.$
\end{proof}

\begin{lemma}[The shrinking balls lemma; Corollary 3.3 in \cite{Topping1}]
\label{nested balls}
Suppose $\left( M^n , g(t) \right)$ is a Ricci flow for $0 \leq t \leq T$ on an $n$-dimensional manifold $M.$ 
Then there exists a $\beta = \beta (n) \geq 1$ such that the following is true. \\
Suppose $x_0 \in M$ and that $\B_{g(0)} (x_0 , r) \subset \subset M$ for some $r>0,$ and 
$| \Rm |_{g(t)} \leq \frac{c_0}{t},$ or more generally $\Ric_{g(t)} \leq (n-1)\frac{c_0}{t},$ 
on $\B_{g(0)}(x_0,r) \cap \B_{g(t)} ( x_0 , r - \beta \sqrt{c_0 t} )$ for each $t \in (0,T]$ and some $c_0 > 0.$ Then for all $0 \leq t \leq T$ 
\beq
\label{initial time} 
	\B_{g(t)}  \left( x_0 , r - \beta \sqrt{c_0 t} \right) \subset \B_{g(0)} (x_0 , r). 
\eeq 
More generally, for $0 \leq s \leq t \leq T,$ we have 
\beq
\label{general time} 
	\B_{g(t)}  \left( x_0 , r - \beta \sqrt{c_0 t} \right) \subset \B_{g(s)} \left(x_0 , r- \beta \sqrt{c_0 s}\right). 
\eeq
\end{lemma}

\begin{lemma}[The expanding balls lemma; see Lemma 3.1 in \cite{Topping1} and Lemma 2.1 in \cite{Topping2}]
\label{expanding balls}
Suppose $K > 0$ and $\left( M , g(t) \right)$ is a Ricci flow for $ t \in [-T,0],$ $T>0,$ on a manifold $M$ of any dimension. Suppose $x_0 \in M$ and that 
$\B_{g(0)} (x_0 , R) \subset \subset M$ and $\Ric_{g(t)} \geq -K$ on $\B_{g(0)}(x_0,R) \cap \B_{g(t)} \left( x_0 , Re^{Kt} \right) \subset \B_{g(t)} \left( x_0 , R \right)$ 
for each $t \in [-T,0].$ Then for all $t \in [-T,0]$ 
\beq
\label{expanding contain} 
\mathbb{B}_{g(0)} \left(x_0 , R\right) \supset \mathbb{B}_{g(t)} \left(x_0 , Re^{Kt}\right). 
\eeq
\end{lemma}

\begin{lemma}[Bi-H\"{o}lder Distance Estimates; Lemma 3.1 in \cite{Topping2}]
\label{bi-holder distance estimates}
Suppose $\left( M^n , g(t) \right)$ is a Ricci flow for $t \in (0,T],$ not necessarily complete, such that for some $x_0 \in M,$ 
and all $t \in (0,T],$ we have $\B_{g(t)} (x_0 , 2r) \subset \subset M.$ 
Suppose further that for some $c_0 , \al > 0,$ and for each $ t \in (0,T],$ we have 
\beq
\label{hyp 1} 
-\alpha \leq  \Ric_{g(t)} \leq \frac{(n-1)c_0}{t} 
\eeq
throughout $\B_{g(t)} (x_0 , 2r).$ 
Define $\Omega_T := \bigcap_{0 < t \leq T} \B_{g(t)} (x_0 , r).$
Then for any $x,y \in \Omega_T$ the distance $d_{g(t)} (x,y)$ is unambigious for all 
$t \in (0,T]$ 
and must be realised by a minimising geodesic lying within 
$\B_{g(t)}(x_0,2r).$ 
Then, for any $0 < t_1 \leq t_2 \leq T,$ we have 
\beq
\label{conc 1} 
	d_{g(t_1)}(x,y) - \beta \sqrt{c_0} \left( \sqrt{t_2} - \sqrt{t_1}\right) \leq d_{g(t_2)} (x,y) \leq e^{\alpha(t_2-t_1)}d_{g(t_1)}(x,y), 
\eeq
where $\beta = \beta(n)>0.$ In particular, $d_{g(t)}$ converges uniformly to a distance metric $d_0$ on $\Omega_T$ as $t \downarrow 0,$ and 
\beq
\label{conc 1'} 
	d_{0}(x,y) - \beta \sqrt{c_0t}  \leq d_{g(t)} (x,y) \leq e^{\alpha t}d_{0}(x,y), 
\eeq
for all $t \in (0,T].$ 
Moreover, there exists $\gamma > 0,$ depending only on $n,$ $c_0$ and upper bounds for $T$ and $r,$ such that 
\beq
\label{conc 2} 
	d_{g(t)} (x,y) \geq \gamma \left[ d_0 (x,y) \right]^{1+2(n-1)c_0} 
\eeq 
for all $t \in (0,T].$ Finally, for all $t \in (0,T]$ and $R < R_0 :=re^{-\alpha T} - \beta \sqrt{c_0 T} < r,$ 
we have 
\beq
\label{conc 3} 
	\B_{g(t)} (x_0 , R_0 ) \subset \Omega_T \quad \text{and} \quad \B_{d_0} (x_0 , R) \subset \subset \mathcal{O} 
\eeq 
where $\mathcal{O}$ is the component of $\text{Interior}\left( \Omega_T \right)$ containing $x_0.$ 
\end{lemma}

\begin{lemma}[Distance function convergence under local convergence; Lemma 6.1 in \cite{Topping2}]
\label{dist funcs under loc conv}
Suppose $\left( M_i , g_i \right)$ is a sequence of smooth $n-$dimensional Riemannian manifolds, possibly incomplete, 
and $x_i \in M_i$ for each $i.$ 
Suppose  there exist a smooth, possibly incomplete $n-$dimensional Riemannian manifold 
$\left( \n , \hat{g} \right)$ 
and a point $x_0 \in \n$ with 
$\B_{\hat{g}}( x_0 , 2r) \subset \subset \n$ 
for some $r>0,$ and a sequence of smooth maps $\varphi_i : \n \rightarrow M_i,$ 
diffeomorphic onto their images, with $\vph_i(x_0)=x_i$, 
such that $\varphi_i^{\ast} g_i \rightarrow \hat{g}$ smoothly on $\overline{\B_{\hat{g}} ( x_0 , 2r)}.$ Then
\begin{enumerate}[$\ $1.]
	\item  If $0 < a \leq 2r,$ and $a < b,$ then $\varphi_i \left( \B_{\hat{g}} \left( x_0 , a \right) \right) \subset \B_{g_i} \left( x_i , b \right)$ for sufficiently large $i.$
	\item  If $0 < a < b \leq 2r,$ then $ \B_{g_i} (x_i , a ) \subset \subset \varphi_i \left( \B_{\hat{g}} \left( x_0 , b \right) \right)$ for sufficiently large $i.$
	\item  For every $ s \in (0,r),$ we have 
	\[ d_{g_i} \left( \varphi_i(x) , \varphi_i(y) \right) \rightarrow d_{\hat{g}}(x,y) \] as $i \rightarrow \infty,$ uniformly for $x$ and $y$ in $\B_{\hat{g}} ( x_0 , s).$
\end{enumerate}
\end{lemma}

\section{Appendix - Shi's Estimates and Compactness}\label{appB}

\noindent
A useful variant of Shi's derivative estimates, that is implicit in Section 5 of \cite{Topping2}, is the following result.

\begin{lemma}[Local Shi decay]
\label{Shi package}
Let $\left( M^n , g(t)\right)$ be a smooth Ricci flow for $t \in \left[ 0 , T \right],$ 
and assume for some $R > 0$ and $x_0 \in M$ that 
$\B_{g(0)} ( x_0 , R ) \subset \subset M.$ 
Moreover, suppose that for all 
$0 < t \leq T$ 
we have 
$|\Rm|_{g(t)} \leq \frac{c_0}{t}$ 
throughout 
$\B_{g(0)} ( x_0 , R)$ 
for some $c_0 > 0.$ 
Then for any $\ep \in \left(0 , R\right),$
there exists 
$\hat{T} = \hat{T} \left( c_0 , n , \varepsilon \right) > 0$ 
and, for $l \in \N,$ there exists 
$C_l = C_l \left( l , c_0 , n , \varepsilon \right) > 0$ 
such that if $0 < \tau \leq \min \{ T , \hat{T} \}$ then we have 
$\B_{g \left( \tau \right)} ( x_0 , R - \ep ) \subset \B_{g(0)} ( x_0 , R)$ and 
\beq
\label{shi_conseq}
\left| \nabla^l \Rm \right|_{g(t)} \leq \frac{C_l}{t^{1+\frac{l}{2}}}
\eeq
throughout $\B_{g(\tau)} ( x_0 , R- \ep ) \times \left(0, \tau \right].$
\end{lemma}

\begin{proof}[Proof of Lemma \ref{Shi package}]
Let $\beta = \beta (n) \geq 1$ be the constant arising in the shrinking balls lemma \ref{nested balls}. 
Define $\hat{T} := \min \left\{ c_0 , \frac{\ep^2} {9 \beta^2 c_0} \right\}> 0$ and let 
$0 < \tau \leq \min \{ T , \hat{T} \}.$ 
The $c_0/t$ curvature bound means that from Lemma \ref{nested balls} we deduce that 
$\B_{g \left( \tau \right) } ( x_0 , R - \varepsilon) \subset \B_{g(0)} \left( x_0 , R - \frac{2\varepsilon}{3} \right) \subset \B_{g(0)} ( x_0 , R ).$ 

Let $ x \in \B_{g \left( \tau \right)} \left( x_0 , R - \varepsilon \right),$ 
$t \in ( 0 , \tau],$ 
and consider $\B_{g \left( \frac{t}{2} \right)} \left( x , \frac{\varepsilon}{3} \right).$ 
We have $ x \in \B_{g(0)} \left( x_0 , R - \frac{2\varepsilon}{3} \right)$, as we have just shown, 
hence via the shrinking balls lemma \ref{nested balls} we have 
$ \B_{g \left( \frac{t}{2} \right)} \left( x , \frac{\varepsilon}{3} \right) \subset \B_{g(0)} ( x , \frac{2\ep}{3} ) \subset \B_{g(0)} ( x_0 , R ) \subset \subset M.$ 

Thus $\overline{ \B_{g \left( \frac{t}{2}\right)} \left( x , \frac{\varepsilon}{4} \right)} \subset \subset \B_{g \left( \frac{t}{2} \right)} \left( x , \frac{\varepsilon}{3} \right)$ and 
$|\Rm|_{g(s)} \leq \frac{2c_0}{t}$ throughout $\B_{g \left( \frac{t}{2} \right)} \left( x , \frac{\varepsilon}{3} \right)$ for all $s\in  \left[ \frac{t}{2} , t \right].$ 
We can apply Theorem 14.14 in \cite{Chow2} to the Ricci flow $s\mapsto g(s + t/2)$ for $s\in [0 , t/2]$,
with $r:= \frac{\varepsilon}{4},$ $K := \frac{2c_0}{t} \geq 1$ and $\alpha := c_0$ to deduce
that for a constant $C = C \left( l , c_0 , n, \ep \right) > 0$ we have 
\beq
	\left| \nabla^l \Rm \right|_{g \left( s + \frac{t}{2} \right)}(x) \leq \frac{2c_0 C}{s^{\frac{l}{2}}t}
\eeq
for all $s\in  \left( 0 , \frac{t}{2} \right].$
Here we have used the observation in \cite{Topping2} that if $K \geq 1$ then, at the central point $x,$ the constant $C \left( \alpha, K , r , m , n\right)$ arising in Theorem 14.14 in \cite{Chow2} 
can be written in the form $C \left(\alpha, r, m , n \right) K.$ 
Restricting  to $s=t/2$ then gives \eqref{shi_conseq} as required.
\end{proof}
%
\cmt{Let $\beta = \beta (n) \geq 1$ be the constant arising in the 
shrinking balls lemma \ref{nested balls}. 
Define $\hat{T} := \frac{\varepsilon^2} {9 \beta^2 c_0} > 0$ and let 
$0 < \tau \leq \min \{ T , \hat{T} \}.$ 
The $c_0/t$ curvature bound means that from Lemma \ref{nested balls} we deduce that 
$\B_{g \left( \tau \right) } ( x_0 , R - \varepsilon) \subset \B_{g(0)} \left( x_0 , R - \frac{2\varepsilon}{3} \right) \subset \B_{g(0)} ( x_0 , R ).$ 

Let $ x \in \B_{g \left( \tau \right)} \left( x_0 , R - \varepsilon \right),$ 
$t \in ( 0 , \tau],$ 
and consider $\B_{g \left( \frac{t}{2} \right)} \left( x , \frac{\varepsilon}{3} \right).$ 
We have $ x \in \B_{g(0)} \left( x_0 , R - \frac{2\varepsilon}{3} \right)$, as we have just shown, 
hence via the shrinking balls lemma \ref{nested balls} we have 
$ \B_{g \left( \frac{t}{2} \right)} \left( x , \frac{\varepsilon}{3} \right) \subset \B_{g(0)} ( x , \frac{2\ep}{3} ) \subset \B_{g(0)} ( x_0 , R ) \subset \subset M.$ 

Thus $\overline{ \B_{g \left( \frac{t}{2}\right)} \left( x , \frac{\varepsilon}{4} \right)} \subset \subset \B_{g \left( \frac{t}{2} \right)} \left( x , \frac{\varepsilon}{3} \right)$ and 
$|\Rm|_{g(s)} \leq \frac{2c_0}{t}$ throughout $\B_{g \left( \frac{t}{2} \right)} \left( x , \frac{\varepsilon}{3} \right)$ for all $s\in  \left[ \frac{t}{2} , t \right].$ 

We can apply Theorem 14.14 in \cite{Chow2} to the Ricci flow 
$s\mapsto g(s-t/2)$ for $s\in [t/2,t]$,
with $r:= \frac{\varepsilon}{4},$ $K := \frac{2c_0}{t}$ and $\alpha := c_0$ to deduce
that for a constant $C = C \left( l , c_0 , n, \ep \right) > 0$ we have 

\beq
\label{>x}
	\left| \nabla^l \Rm \right|_{g(s)}(x) \leq \frac{2c_0 C}{(s-t/2)^{\frac{l}{2}}t}
\eeq

throughout $\B_{g\left( \frac{t}{2} \right)} \left( x , \frac{\varepsilon}{8} \right)$ 
for all $s\in  \left( \frac{t}{2} , t \right].$

Here we have used the observation in \cite{Topping2} that the constant $C \left( \alpha, K , r , m , n\right)$ arising in Theorem 14.14 in \cite{Chow2} 
can be written in the form $C \left(\alpha, r, m , n \right) K.$ 
Restricting  to $s=t$ then gives \eqref{shi_conseq} as required.
}


\noindent
The following lemma localises the well-known Hamilton-Cheeger-Gromov compactness lemma.
The proof carries over more or less verbatim, and details can be found in \cite{mcleod_phd}.

\begin{lemma}[Local compactness]
\label{loc_compactness_manifolds}
Suppose $(M_i,g_i)$ is a sequence of smooth $n$-dimensional Riemannian manifolds, not necessarily complete,
and that $x_i\in M_i$ for each $i$. 
Suppose that, for some $R,v>0$, we have  $\B_{g_i}(x_i,R)\subset\subset M_i$ 
and $\VolBB_{g_i}(x_i,R)\geq v$ for each $i$, and that (for each $l$) we have
$|\grad^l\Rm|_{g_i}\leq C$ throughout $\B_{g_i}(x_i,R)$, where $C$ is independent of $i$,
but allowed to depend on $l$.

Then after passing to an appropriate subsequence in $i$, there exist a smooth, typically incomplete $n$-dimensional Riemannian manifold 
$(\n,g_\infty)$, 
a point $x_0\in \n$ with 
$\B_{g_\infty}(x_0,r)\subset\subset\n$ for every $r\in (0,R)$, 
and a sequence of smooth maps $\vph_i:\B_{g_\infty}(x_0,\frac{i}{i+1}R)\to M_i$, 
diffeomorphic onto their images and mapping $x_0$ to $x_i$, such that 
$\vph_i^*g_i \to g_\infty$ smoothly locally on
$\B_{g_\infty}(x_0,R)$.
\end{lemma}

\bcmt{Example: sequence of flat balls of radius $R+1/i$. Really want $\n$ to be flat ball of radius $R$, so don't want to ask that $\B_{g_\infty}(x_0,R)\subset\subset\n$.}

\bcmt{In proof, it's enough to ask that for all $r\in (0,R)$, we can find sequence of diffeos from some manifold $\n$ with a metric $g_\infty$ for which $\B_{g_\infty}(x_0,r)\subset\subset\n$ to $M_i$, so that $\vph_i^*g_i \to g_\infty$ smoothly uniformly on $\B_{g_\infty}(x_0,r)$. The manifold $\n$ of the lemma then arises a little as in our paper, taking disjoint union and quotient etc. as in Lemma \ref{smooth manifold construction}. }

\noindent
The local version of Hamilton-Cheeger-Gromov compactness for flows, which is already implicit in \cite{Topping2}, is the following.

\begin{lemma}[Local Ricci flow compactness]
\label{loc_compactness_flows}
Suppose $(M_i,g_i(t))$ is a sequence of smooth $n$-dimensional Ricci flows, not necessarily complete, each defined for $t\in [0,T]$, and that $x_i\in M_i$ for each $i$. 
Suppose that, for some $R>0$, we have $\B_{g_i(0)}(x_i,R)\subset\subset M_i$ for each $i$, that
$\VolBB_{g_i(0)}(x_i,R)\geq v>0$, 
and throughout $\B_{g_i(0)}(x_i,R)$
that $\Ric_{g_i(t)}\geq -\al<0$ for all $t\in [0,T]$ and  
$|\Rm|_{g_i(t)}\leq c_0/t$ for all $t\in (0,T]$, for positive constants $v$, $\al$ and $c_0$ that are independent of $i$.

Then for all $\eta\in (0,R/2)$, there exists $S>0$ depending only on $R$, $n$, $v$, $\al$, $c_0$ and $\eta$
such that 
after passing to an appropriate subsequence in $i$, there exist a smooth $n$-dimensional manifold $\n$, a point $x_0\in \n$ 
and a Ricci flow $g(t)$ on $\n$ for $t\in (0,\tau]$,
where $\tau:=\min\{T,S\}$, with the following properties.

First, $\B_{g(t)}(x_0,R-\eta)\subset\subset\n$ for all 
$t\in (0,\tau]$. Second, if we define
$\Om$ to be the connected component of the interior of 
$$\bigcap_{s\in (0,\tau]} \B_{g(s)}(x_0,R-\eta)\subset \n$$
containing $x_0$,
then for all $t\in (0,\tau]$ we have
$\B_{g(t)}(x_0,R-2\eta)\subset \Om$.
Third, there exists a sequence of smooth maps 
$\vph_i:\Om\to \B_{g_i(0)}(x_i,R)\subset M_i$, diffeomorphic onto their images and mapping $x_0$ to $x_i$, such that 
$\vph_i^*g_i(t) \to g(t)$ smoothly uniformly on
$\Om\times [\de,\tau]$
for every $\de\in (0,\tau)$.

Finally, throughout $\Om$ we have
$\Ric_{g(t)}\geq -\al$  and  
$|\Rm|_{g(t)}\leq c_0/t$ for all $t\in (0,\tau]$.
\end{lemma}

\bcmt{Be aware that $\vph_i$ as written currently may map outside the region where our curvature estimates hold. This is no problem getting the the curvature estimates passed to the limit as it turns out.}

\begin{proof}
We begin by applying the \emph{local Shi decay lemma}
\ref{Shi package} to each $g_i(t)$, with $\ep:=\eta/3$.
This ensures that there exists $S>0$ depending 
only on $n$, $c_0$ and $\eta$ such that for 
$0<t\leq \min \{ T , S \} $ we have 
$\B_{g_i(t)} \left( x_i , R-\ep \right) \subset \B_{g_i(0)} \left( x_i , R \right)$
and that for 
$0<t\leq \tau\leq \min \{ T , S \} $ we have
\beq
\left| \nabla^l \Rm \right|_{g_i(t)} \leq \frac{C_l}{t^{1+\frac{l}{2}}}
\eeq
throughout $\B_{g_i(\tau)} \left( x_i , R-\ep \right)$,
where $C_l$ depends on $l$, $c_0$, $n$ and $\eta$.
Next with a view to later applying the expanding and shrinking balls lemmas, 
we reduce $S>0$ further, depending now also on $\al$, so that
\beq
\label{S_constraint}
R(1-e^{-\al S})<\ep\qquad \text{and}\qquad S\leq \frac{\eta^2}{\be^2 c_0}.
\eeq
where $\be=\be(n)\geq 1$ comes from Lemma \ref{nested balls}.
A final reduction of $S>0,$ depending now also on $v,$ ensures that by Lemma 2.3 of \cite{Topping1} (which is in a more appropriate form than the variant Lemma \ref{Volume 2})
and Bishop-Gromov, we have 
$\VolBB_{g_i(s)}(x_i,R-\ep)\geq \ti v>0$ for all 
$s\in [0,\min\{T,S\}]$,
where $\tilde{v}$ depends only on $v, \al, R$ and $n.$

At this point we fix $S$, and the corresponding $\tau:=\min\{T,S\}$, and apply the local compactness lemma \ref{loc_compactness_manifolds} to the sequence $g_i(\tau)$
with $R$ there equal to $R-\ep$ here.
The conclusion is
that after passing to an appropriate subsequence in $i$, there exist a smooth $n$-dimensional Riemannian manifold $(\n,g_\infty)$, a point $x_0\in \n$ with $\B_{g_\infty}(x_0,r)\subset\subset\n$ for every $r\in (0,R-\ep)$, 
and a sequence of smooth maps $\vph_i:\B_{g_\infty}(x_0,\frac{i}{i+1}(R-\ep))\to M_i$, diffeomorphic onto their images and mapping $x_0$ to $x_i$, such that 
$\vph_i^*g_i(\tau) \to g_\infty$ smoothly locally on
$\B_{g_\infty}(x_0,R-\ep)$.

By Part 1 of Lemma \ref{dist funcs under loc conv}, applied with $g_i$ and 
$\hat g$ there equal to $g_i(\tau)$ and $g_\infty$ here, respectively, and with $a=2r$ and $b$ there equal to $R-2\ep$ 
and $R-\ep$ here, respectively, we find that after dropping finitely many terms, we have
$$\varphi_i \left( \B_{g_\infty} \left( x_0 , R-2\ep \right) \right) \subset \B_{g_i(\tau)} \left( x_i , R-\ep \right)
\subset \B_{g_i(0)} \left( x_i , R \right)$$ 
for every $i$ (where the second inclusion here was established at the beginning of the proof).


By Hamilton's original argument \cite{Ham95} we can pass to a further subsequence and find a Ricci flow $g(t)$
on $\B_{g_\infty}(x_0,R-\ep)$, $t\in (0,\tau]$ so that
$g(\tau)=g_\infty$ on $\B_{g_\infty}(x_0,R-\ep)$
and so that
$\vph_i^*g_i(t) \to g_\infty$ smoothly locally on
$\B_{g_\infty}(x_0,R-\ep)\times (0,\tau]$ as $i\to\infty$.
In particular, we can pass our curvature hypotheses to the limit to obtain that
$\Ric_{g(t)}\geq -\al$ and  
$|\Rm|_{g(t)}\leq c_0/t$, for all $t\in (0,\tau]$ and  
throughout $\B_{g_\infty}(x_0,R-2\ep)$.

Next, our constraint \eqref{S_constraint}
implies that $(R-2\ep)(1-e^{-\al S})<\ep$, i.e. that 
$R-3\ep<(R-2\ep)e^{-\al S}$, and hence by the expanding balls lemma \ref{expanding balls}, we know that for all $t\in (0,\tau]$ we have
$$\n\supset\supset \B_{g(\tau)}(x_0,R-2\ep)\supset \B_{g(t)}(x_0,(R-2\ep)e^{\al (t-\tau)}) \supset \B_{g(t)}(x_0,(R-2\ep)e^{-\al S}) \supset \B_{g(t)}(x_0,R-3\ep),$$
and hence (recalling that $\ep=\eta/3$) we have $\B_{g(t)}(x_0,R-\eta)\subset \B_{g(\tau)}(x_0,R-2\ep)
\subset\subset\n$ as required.
One consequence is that our curvature bounds hold within each 
$\B_{g(t)}(x_0,R-\eta)$, for all $t\in (0,\tau]$.
Moreover, if we reduce $\n$ to $\B_{g_\infty}(x_0,R-\ep)\subset \n$,
then we still have $\B_{g(t)}(x_0,R-\eta)\subset\subset\n$ for all
$t\in (0,\tau]$, and now the Ricci flow is defined throughout $\n$.

It remains to show that $\B_{g(t)}(x_0,R-2\eta)\subset \Om$,
and for that it suffices to prove that
\beq
\label{desired_balls_nesting}
\B_{g(t)}(x_0,R-2\eta)\subset \B_{g(s)}(x_0,R-\eta)\qquad
\text{for each } s,t\in (0,\tau].
\eeq
In the case $s<t$ this follows from the shrinking balls lemma \ref{nested balls}
applied with time zero there equal to time $s$ here, and $r$ there equal to $R-\eta$ here.
That lemma tells us that 
$\B_{g(t)}(x_0,R-\eta-\be\sqrt{c_0(t-s)})\subset \B_{g(s)}(x_0,R-\eta)$,
by \eqref{initial time} not \eqref{general time}, and because
$\be\sqrt{c_0(t-s)}\leq \be\sqrt{c_0 S}\leq \eta$, by \eqref{S_constraint}, 
this implies \eqref{desired_balls_nesting}.

Meanwhile, in the case $s>t$, \eqref{desired_balls_nesting} follows from the expanding balls lemma \ref{expanding balls} 
applied with time zero there equal to time $s$ here, and $R$ there equal to $R-\eta$ here.
That lemma tells us that 
$\B_{g(s)}(x_0,R-\eta) \supset \B_{g(t)}(x_0,(R-\eta)e^{\al(t-s)})$, and so we will have proved  \eqref{desired_balls_nesting} if we can prove that 
$(R-\eta)e^{\al(t-s)}\geq R-2\eta$, or more generally that 
$(R-\eta)e^{-\al S}\geq R-2\eta$, which is equivalent to 
$(R-\eta)(1-e^{-\al S})\leq \eta$. This is turn follows from the first part of 
\eqref{S_constraint}.
\end{proof}

\section{Appendix - Smooth Manifold Construction}\label{appC}

\begin{theorem}[Smooth Manifold Construction]
\label{smooth manifold construction}
Assume that for each $i \in \N$ we have a smooth $n$-manifold $M_i$ and a point $x_i \in M_i,$ and that each $M_i$ is contained in the next in the sense that for each $i\in\N$ there exists a smooth map 
$\psi_i : M_i \rightarrow M_{i+1},$ 
mapping $x_i$ to $x_{i+1}$ and diffeomorphic onto its image. 
Then there exists a smooth $n$-manifold $M,$ 
containing a point $x_0,$ and there exist smooth maps 
$\theta_i : M_i \rightarrow M,$ 
all mapping $x_i$ to $x_0,$ diffeomorphic onto their image, and satisfying that 
$ \theta_{i} ( M_i ) \subset \theta_{i+1} ( M_{i+1} ),$ 
and further that
\beq
\label{useful form of manifold} 
	M = \bigcup_{i=1}^{\infty} \theta_i ( M_i ). 
\eeq
Moreover, we have that 
\beq
\label{theta relate psi 1} 
	\psi_i=\theta_{i+1}^{-1} \circ \theta_i : M_i \rightarrow M_{i+1}. 
\eeq 
\end{theorem}

\vskip 2pt

\begin{proof}
Define $M := \bigsqcup_{i=1}^{\infty} M_i \big/ \sim $, equipped with the quotient topology,
where $\sim$ is the equivalence relation generated by identifying points $x$ and $y$ if $y = \psi_i (x)$ for some $i \in \N.$ Let $x_0 \in M$ be the equivalence class generated by the points $x_i \in M_i.$
For each $i \in \N$ define $\theta_i:M_i\to M$ to be the map sending a point $x$ to the equivalence class $[x]$. Thus $\th_i(x_i)=x_0$, and $\theta_i$ is a homeomorphism onto its image, while
$M = \bigcup_{i=1}^{\infty} \theta_i ( M_i )$
which is \eqref{useful form of manifold}.
Moreover, for each $x\in M_i$ we have 
$\th_i(x)=[x]=[\psi_i(x)]=\th_{i+1}(\psi_i(x))$, which gives
\eqref{theta relate psi 1}.
Since $\psi_i$ is a diffeomorphism onto its image, \eqref{theta relate psi 1}
allows us to combine the smooth atlases for each $M_i$ into a smooth atlas for $M$ by composing with the maps $\theta_i^{-1}.$
Hence we simultaneously establish both that $M$ is a smooth $n$-manifold, and that each $\theta_i$ is a diffeomorphism onto its image as claimed.
\end{proof}

\noindent
{\sc Mathematics Institute, University of Warwick, Coventry,
CV4 7AL, UK} 

\vskip 5pt

\noindent
AM:
\emph{
a.d.mcleod@warwick.ac.uk} 

\noindent
\url{http://warwick.ac.uk/fac/sci/maths/people/staff/mcleod/} 


\vskip 5pt

\noindent
PT:\\
\url{http://homepages.warwick.ac.uk/~maseq/}

\end{document}